\documentclass[12pt,a4paper,psamsfonts]{amsart}

\usepackage{fancyhdr}
\usepackage{appendix}
\usepackage{amssymb,amscd,amsxtra,calc}
\usepackage{mathrsfs}
\usepackage{mathtools}
\usepackage{cmmib57}
\usepackage{multirow}
\usepackage[all]{xy}
\usepackage{longtable}
\usepackage[colorlinks,linkcolor=blue,anchorcolor=blue,citecolor=green]{hyperref}

\theoremstyle{plain}
    \newtheorem{thm}{Theorem}[section]

     \newtheorem{conjecture}[thm]{Conjecture}
    \newtheorem{corollary}[thm]{Corollary}
    
    \newtheorem{lemma}[thm]{Lemma}
    \newtheorem{proposition}[thm]{Proposition}
    \newtheorem{question}[thm]{Question}
    \newtheorem{theorem}[thm]{Theorem}

\theoremstyle{definition}
    \newtheorem{definition}[thm]{Definition}
    \newtheorem{notation}[thm]{Notation}
    \newtheorem*{notation*}{Notation and Terminology}
    \newtheorem{remark}[thm]{Remark}
    \newtheorem*{ack}{Acknowledgments}
\theoremstyle{remark}

\newcommand{\Q}{\mathbb{Q}}

\newcommand{\R}{\mathbb{R}}

\newcommand{\Z}{\mathbb{Z}}

\newcommand{\alb}{\operatorname{alb}}

\newcommand{\Aut}{\operatorname{Aut}}

\newcommand{\ch}{\operatorname{char}}

\newcommand{\Fix}{\operatorname{Fix}}

\newcommand{\GL}{\operatorname{GL}}

\newcommand{\id}{\operatorname{id}}

\newcommand{\Ker}{\operatorname{Ker}}

\newcommand{\NS}{\operatorname{NS}}

\newcommand{\Per}{\operatorname{Per}}

\newcommand{\Sing}{\operatorname{Sing}}

\newcommand{\Supp}{\operatorname{Supp}}

\newcommand{\N}{\operatorname{N}}

\newcommand{\Alb}{\operatorname{Alb}}

\newcommand{\Pic}{\operatorname{Pic}}

\newcommand{\red}{\mathrm{red}}

\begin{document}

\title[Invariant subvarieties]
{Invariant subvarieties with small dynamical degree}

\address{Department of Mathematics, box 1917, Brown University, Rhode Island 02912, USA}
\email{matsuzawa@math.brown.edu}

\address{Korea Institute For Advanced Study, 
Dongdaemungu Seoul 02455, Republic of Korea}
\email{ms@u.nus.edu, shengmeng@kias.re.kr}

\author{Yohsuke Matsuzawa, Sheng Meng, Takahiro Shibata, De-Qi Zhang, Guolei Zhong}

\address
{
\textsc{National University of Singapore,
Singapore 119076, Republic of Singapore
}}
\email{mattash@nus.edu.sg}

\address
{
\textsc{National University of Singapore,
Singapore 119076, Republic of Singapore
}}
\email{matzdq@nus.edu.sg}

\address
{
\textsc{National University of Singapore,
Singapore 119076, Republic of Singapore
}}
\email{zhongguolei@u.nus.edu}

\begin{abstract}
Let $f:X\to X $ be a dominant self-morphism of an algebraic variety.
Consider the set $\Sigma_{f^{\infty}}$ of $f$-periodic subvarieties of small dynamical degree, the subset $S_{f^{\infty}}$ of maximal elements in $\Sigma_{f^{\infty}}$, and the subset $S_f$ of $f$-invariant elements in $S_{f^{\infty}}$. When $X$ is projective, we prove the finiteness of the set $P_f$ of $f$-invariant prime divisors with small dynamical degree, and give an optimal upper bound 
$$\sharp P_{f^n}\le d_1(f)^n(1+o(1))$$ 
as $n\to \infty$, where $d_1(f)$ is the first dynamic degree.
When $X$ is an algebraic group (with $f$ being a translation of an isogeny), or a (not necessarily complete) toric variety, we give an optimal upper bound 
$$\sharp S_{f^n}\le d_1(f)^{n\cdot\dim(X)}(1+o(1))$$ as $n \to \infty$, which slightly generalizes  a conjecture of S.-W. Zhang for polarized $f$.
\end{abstract}



\maketitle
\tableofcontents

\section{Introduction}

We work over a fixed algebraically closed field $k$.
Throughout the paper, we assume the characteristic of $k$ is zero unless otherwise indicated.

Let $f:X\to X $ be a dominant self-morphism of an algebraic variety $X$.
Let $V\subseteq X$ be an (irreducible) closed subvariety.
We say that $V$ is \textit{$f$-periodic} (resp.\,\textit{$f$-invariant}) if $\overline{f^s(V)}=V$ for some $s>0$ (resp.\,$\overline{f(V)}=V$). 
When $X$ is projective or $f$ is finite (and hence surjective), $\overline{f^s(V)}=V$ implies $V$ is $f$-periodic in the usual sense, i.e., $f^s(V)=V$. 
We say that $V$ is \textit{$f^{-1}$-periodic}  if $f^{-s}(V)=V$ for some $s>0$. 

An $f$-periodic closed subvariety $V$ is said to be of \textit{Small Dynamical Degree} 
({\it SDD} for short) if the first dynamical degree $d_1(f^s|_V)<d_1(f^s)$ for some $s>0$ such that $\overline{f^s(V)}=V$.  
Denote by $\Sigma_{f^{\infty}}$ the set of $f$-periodic (irreducible closed) subvarieties of small dynamical degree.
For $V\in \Sigma_{f^{\infty}}$, we say that $V$ is {\it maximal} if $V\subseteq V'$ and $V'\in \Sigma_{f^{\infty}}$ imply $V=V'$.

Denote by $S_{f^{\infty}}$ the set of maximal 
elements in $\Sigma_{f^{\infty}}$.
Denote by 
$$S_f \hskip 1pc (\textup{resp.}\,\,\, P_f)$$ 
the subset of 
$S_{f^{\infty}}$, consisting of all $f$-invariant elements (resp.\,prime divisors).
In general, $\Sigma_{f^{\infty}}$ may have infinitely many non-maximal elements. So the question below is natural.

\begin{question}\label{Q1}
Let $f: X \to X$ be a dominant self-morphism of an algebraic variety $X$.
\begin{itemize}
\item[(1)]
Is $S_f$ a finite set? Namely, are there only finitely many $f$-invariant maximal $f$-periodic subvarieties of small dynamical degree?

\item[(2)]
Is $P_f$ a finite set? Namely, are there only finitely many prime divisors in $S_f$?
\end{itemize}
\end{question}

\begin{remark}
Question \ref{Q1} was first proposed in \cite[Question 8.4\,(3)]{MMSZ20} by the first four authors and the motivation is to study the points with small arithmetic degrees.
Indeed, together with positive answers to other two arithmetic questions in \cite[Question 8.4]{MMSZ20} (cf.~Question \ref{ques_subvar}), it is easy to show the equivalence of the (weak) Kawaguchi-Silverman Conjecture (cf.~\cite{KS16}, Conjecture \ref{conj_ks}) and the (strong) small Arithmetic Non-Density (sAND) conjecture (cf.~\cite[Conjecture 1.4]{MMSZ20} = Conjecture \ref{conj_zf}); see Theorem \ref{thm-equiv}.
In this paper, we will mainly focus on Question \ref{Q1} itself.
We refer to \cite[Theorem 2.33]{LS19} for periodic SDD subvarieties of Hyperk\"ahler varieties.
\end{remark}

When studying many dynamical problems, we usually deal with the periodic subvarieties instead of simply invariant ones to observe some ``growth" property of $f$.
In general, the set $S_{f^{\infty}}$ may not be finite.
For example, when $X$ is projective and $f$ is $q$-polarized with $q>1$, the set $S_{f^{\infty}}$ coincides with the set of {\it periodic points} $\Per(f)$, which is Zariski dense and countable by \cite[Proof of Theorem 5.1]{Fak03}.
Therefore, it is reasonable to propose the following question.

\begin{question}\label{Q2}
Let $f: X \to X$ be a dominant self-morphism of an algebraic
variety.
Does the upper bound (of cardinality)
$$\sharp S_{f^n}\le (d_1(f))^{n\cdot\dim(X)}(1+o(1))$$ hold as $n\to \infty$?
\end{question}

\begin{remark}\label{rmk-swz}
(1)
Question \ref{Q2} is sort of a generalization of \cite[Conjecture 1.2.2]{Zha06} which claims that if $X$ is projective and $f$ is $q$-polarized, then $\sharp \Fix(f^n)=q^{n\cdot\dim(X)}(1+o(1))$ as $n\to \infty$.
In this case, the set of fixed points $\Fix(f^n)$ is just $S_{f^{n}}$ and Question \ref{Q2} has a positive answer when $X$ is further smooth and general fixed points have multiplicity one; see the argument after \cite[Conjecture 1.2.2]{Zha06} by using Lefschetz fixed point theorem and Serre's Theorem, i.e., \cite[Theorem 1.1.2]{Zha06} (which are available for smooth $X$, but there are no easily applicable results even for singular surfaces) and noting that $\Fix(f^n)$ is a finite set for each $n \ge 1$.

(2)
Here is a concrete example showing that the upper bound of the inequality in Question \ref{Q2} is optimal.
Let $A$ be an abelian variety and $f:=(-q)_A$ the multiplication map by $-q$ (with $q>1$) on $A$.
Then $f$ is $q^2$-polarized and $d_1(f)=q^2$. 
Hence $\sharp S_{f^{2n+1}}=\sharp\textup{Fix}(f^{2n+1})$ is the number of $(q^{2n+1}+1)$-torsion points.
So we have 
$$(q^2)^{(2n+1)\dim(A)}<\sharp S_{f^{2n+1}}=(q^{2n+1}+1)^{2\dim(A)}=(q^2)^{(2n+1)\dim(A)}(1+o(1))$$ 
as $n\to\infty$. 
\end{remark}

Theorems \ref{main-thm-gp}, \ref{main-thm-pf} and \ref{main-thm-all} are our main results.

The key idea of our first result concerning isogenies (and their translations) of algebraic groups is to find a {\it unique} maximal periodic SDD subgroup; see Definition \ref{def-Gf}.
We leave it to the readers to consider more general cases, e.g., spherical varieties, (almost) homogeneous varieties.
We use the fact $d_1(f^n)=d_1(f)^n$ for any $n>0$ to get the following upper bound when applying theorems in Section \ref{sec-5}.

\begin{theorem}\label{main-thm-gp}
Let $f$ be a dominant self-morphism of a variety $X$.
Then
$$\sharp S_{f^n}\le (d_1(f))^{n\cdot\dim(X)}(1+o(1))$$ holds as $n\to \infty$,
if $(X, f)$ fits one of the following cases.

\begin{itemize}
\item[(1)]
(cf.~Theorem \ref{thm_alg_group}) $X$ is an algebraic group; and $f=R_a\circ g$ where $g$ is an isogeny and $R_a$ is the right multiplication map by some $a\in X$  (when $X$ is semi-abelian, the self-morphism $f$ is automatically of this form).

\item[(2)]
(cf.~Theorem~\ref{thm-toric})
$X$ is a toric variety with the big torus $T$ being $f$-invariant.
\end{itemize}
\end{theorem}

Theorem \ref{main-thm-pf} 
below is a simpler version of
Theorem~\ref{thm_divisor_case}, and
answers Question \ref{Q1}\,(2) affirmatively: the finiteness of SDD prime divisors.
The main idea is to apply the higher dimensional Hodge index theorem to construct sufficiently many pairwise disjoint prime divisors on a smooth model.
In this way, we are able to construct an equivariant fibration to a curve which contracts those  periodic SDD prime divisors; and then what left is simply to count the fixed points of a polarized endomorphism on a curve.

\begin{theorem}\label{main-thm-pf} (cf.~Theorem~\ref{thm_divisor_case})
Let $f$ be a surjective endomorphism of a projective variety. Then
$$\sharp P_{f^n}\le (d_1(f))^n(1+o(1))$$
holds as $n\to \infty$.
\end{theorem}

The upper bound of the inequality in Theorem \ref{main-thm-pf} above is optimal.
Indeed, let $X=Y\times E$, where $Y$ is any projective variety and $E$ is an  elliptic curve   and let $f=\text{id}_{Y}\times f_E$ where $f_E := (-q)_E$ is the multiplication map by $-q$ (with $q>1$) on $E$.	
In this case, by the product formula (cf.~Notation \ref{not-2.1}), $d_1(f) = q^2$, and
a prime divisor $P\in  P_{f^n}$ if and only if it is an $f^n$-stable fibre of the natural projection $X\to E$.
Thus $\sharp P_{f^{2n+1}} = \sharp \Fix(f_E^{2n+1}) = (q^{2n+1}+1)^{2} = (d_1(f))^{2n+1}(1+o(1))$
(cf.~Remark \ref{rmk-swz}).

When $X$ is projective, Question \ref{Q1}\,(2) can be reformulated as follows.

\begin{theorem}\label{main-thm-equiv}
Let $f: X \to X$ be a surjective endomorphism of a projective variety $X$. 
Then the following are equivalent.

\begin{enumerate}
\item  The set $P_{f^\infty}:=\bigcup_{n=1}^{\infty} P_{f^n}$ is infinite.

\item  Replacing $X$ by its normalization, there is an (iteration of) $f$-equivariant fibration $\pi:X\to C$ onto a smooth projective curve $C$ such that $g:=f|_C$ is $d_1(f)$-polarized and the relative dynamical degree  $d_1(f|_\pi)<d_1(f)$.
\end{enumerate}
In particular, the set of $f^{-1}$-periodic  prime divisors of small dynamical degree is finite.
\end{theorem}

Corollary \ref{cor:surf} answers Question \ref{Q1} affirmatively  for projective surfaces, 
which is a direct consequence of Theorem \ref{main-thm-pf} and Proposition \ref{prop_fixed_points}.
However, we are currently unable to show the upper bound in Question \ref{Q2} in full generality except cases in Theorem \ref{main-thm-gp}.

\begin{corollary}\label{cor:surf}
$S_f$ is finite for any surjective endomorphism $f$ of a projective surface.
\end{corollary}

Throughout this paper, we try to make our argument as algebraic as we can so that most of our results can be generalized to positive characteristic.
We leave such generalization to the careful readers which will not be mentioned too much in this paper for the sake of a comfortable reading.

We end the introduction with Theorem~\ref{main-thm-all}, as an application of a technique used when considering algebraic groups (the commutative diagram in Proposition \ref{prop-gen-d1}).
This compound theorem has its own interest and extends the results in \cite[Theorem 1.2 and Theorem 5.1]{CMZ20}, \cite[Proposition 3.3]{MZ19},  \cite[Lemma 2.3]{NZ10} and \cite[Remark 5.9]{San17}
to the most general setting (without assumptions like smooth, characteristic $0$, separable, int-amplified, etc.). Theorem~\ref{main-thm-all} is a simpler version of Theorem~\ref{thm-all}.

\begin{theorem}\label{main-thm-all}(cf.~Theorem \ref{thm-all})
Let $f:X\to X$ be a surjective endomorphism of a normal projective variety over the field $k$ of arbitrary characteristic. Assume $d_1(f)>1$.
Then we have:
\begin{enumerate}

\item 
$f^*D\sim_{\R} d_1(f)D$ for some nef $\R$-Cartier divisor $D$.

\item 
Suppose $f^*B\equiv qB$ for some big divisor and $q>1$. Then $f$ is $q$-polarized.

\item 
Suppose $f$ is $q$-polarized. 
Then any eigenvalue of $f^*|_{H^1_{\textup{\'et}}(X,\Q_\ell)}$ has modulus $\sqrt{q}$.

\item 
Suppose $f$ is $q$-polarized (resp.~int-amplified). Then so is $f|_{\Alb(X)}:\Alb(X)\to \Alb(X)$.
\end{enumerate}
\end{theorem}

In Remark \ref{rmk-swz}, the method of counting fixed points for polarized endomorphisms only holds in characteristic 0, because it applied Serre's theorem to compute the eigenvalues of $f^*|_{H^i(X,\mathbb{Z})}$ which heavily relied on a K\"ahler form.
Serre's theorem is also a K\"ahler analogue of Weil’s conjecture on eigenvalues of the pullback of (the purely inseparable $q$-polarized) Frobenius endomorphism $F$ on cohomology.
It was proved by Deligne's celebrated theorem: all eigenvalues of $F^*|_{H^i_\textup{\'et}(X,\Q_\ell)}$ have modulus $q^{i/2}$ (cf.~\cite[Th\'eor\`eme 1.6]{Del74}).
By our Theorem \ref{main-thm-all}\,(3), this holds for every polarized endomorphism when $i=1$.

\begin{ack}
The authors are supported by a JSPS Overseas Research Fellowship, a Research Fellowship of KIAS, a Research Fellowship of NUS,  an ARF of NUS and a President's Scholarship of NUS, respectively. 
The authors would like to thank Fei Hu, Joseph H. Silverman and Junyi Xie for many inspiring discussions and suggestions.
The authors would also like to thank the referee for many suggestions to improve the paper.
\end{ack}

\section{Preliminaries}\label{section_preliminary}

\begin{notation}\label{not-2.1}
We use the following notation throughout this paper.

\begin{itemize}
\item Let $X$ be a projective variety. Denote by $\textup{NS}(X):=\Pic(X)/\Pic^\circ(X)$ the usual N\'eron-Severi group (in the sense of Cartier divisors).
Let $\textup{N}^1(X):=\textup{NS}(X)\otimes_\mathbb{Z}\mathbb{R}$,  
and $\textup{PE}^1(X)$  the closed cone of pseudo-effective classes in $\textup{N}^1(X)$. 

\item An $r$-cycle $C$ on a projective variety $X$ is {\it pseudo-effective} if $C \cdot H_1 \cdots H_r \ge 0$ for all ample divisors $H_i$. 
Two $r$-cycles $C_1$ and $C_2$ are said to be {\it weakly numerically equivalent}  if
$(C_1 - C_2) \cdot L_1 \cdots L_r= 0$ for all
Cartier divisors $L_i$; cf.~\cite[Section 2]{MZ18b} and the references therein for more information.

\item The symbols $\sim$ (resp.\,$\sim_{\mathbb Q}$,\,$\sim_{\mathbb R}$) and
$\equiv$ (resp.\,$\equiv_w$) denote
the linear equivalence (resp.\,$\mathbb Q$-linear equivalence,\,$\mathbb R$-linear equivalence) and the numerical equivalence (resp.\,weak numerical equivalence) on divisors.

\item A surjective endomorphism $f:X\to X$ of a projective variety is said to be {\it $q$-polarized} if $f^*H\sim qH$ for some ample Cartier divisor $H$ and integer $q>1$.
The unique endomorphism on a point is $q$-polarized for any $q>1$ by convention.
\item A surjective endomorphism $f:X\to X$ of a projective variety is said to be {\it int-amplified} if $f^*H-H=L$ for some ample Cartier divisors $H$ and $L$.

\item Let $X$ be a normal projective variety. Denote by $\textup{Cl}(X)$  the Weil divisor class group on $X$, and by $\textup{Cl}^\circ(X)\subseteq \textup{Cl}(X)$ the subgroup of divisors algebraically equivalent to $0$, which is a birational invariant (cf.~\cite[Lemma 18]{Kol18}). 
The quotient $$\textup{Cl}^{\textup{ns}}(X):=\textup{Cl}(X)/\textup{Cl}^\circ(X)$$ is called the \textit{N\'eron-Severi} class group of $X$ (in the sense of Weil divisors).

There are natural inclusions $\textup{Pic}^\circ(X)\subseteq \textup{Cl}^\circ(X)$ and $\textup{NS}(X)\subseteq\textup{Cl}^{\textup{ns}}(X)$. Both of these  are equalities when $X$ is smooth. 

\item Denote by 
$$\rho(X):=\dim_{\R} \textup{NS}(X) \otimes_{\Z} \R$$ 
the usual {\it Picard number} of $X$.
Denote by
$$\rho^{\textup{ns}}(X):=\dim_{\R} \textup{Cl}^\textup{ns}(X) \otimes_{\Z} \R$$
which is finite (cf.~\cite[Theorem 17]{Kol18}).
We refer readers to \cite[Section 3]{Kol18} and \cite[Chapter 10]{Ful98} for more information.

\item Let $f:X\dashrightarrow X$ and $g:Y\dashrightarrow Y$ be dominant self-maps of projective varieties.
Let $\pi:X\dashrightarrow Y$ be a dominant rational map such that $g\circ\pi=\pi\circ f$.
We recall the following definitions and facts about the first dynamical degree.
\begin{enumerate}

\item 
The \textit{(first) dynamical degree} of $f$ is defined by
$$d_1(f):= \lim_{n \to \infty} ((f^n)^*H \cdot H^{\dim (X) -1})^{1/n}$$
where $H$ is a nef and big Cartier divisor of $X$.

\item 
The \textit{relative (first) dynamical degree} of $f$ is defined by
$$d_1(f|_{\pi})= \lim_{n \to \infty} ((f^n)^*H_X \cdot (\pi^*H_Y)^{\dim (Y)}\cdot H_X^{\dim (X)-\dim(Y) -1})^{1/n}$$
where $H_X$ and $H_Y$ are respectively nef and big Cartier divisors of $X$ and $Y$.

\item Product formula: $d_1(f)=\max\{d_1(f|_{\pi}), d_1(g)\}$.

\item $d_1(f^n)=d_1(f)^n$ for any $n>0$.

\item $d_1(f)=\rho(f^*|_{\N^1(X)})$, the spectral radius of $f^*|_{\N^1(X)}$, if $f$ is a self-morphism.
\end{enumerate}
(Relative) dynamical degrees can also be defined, 
by the above formulas, for (not necessarily projective) varieties over the field $k$ (allowed to be of arbitrary characteristic) by taking any birational projective models since they are birational invariants and independent of the choices of nef and big divisors above; indeed, $d_1(f)$ is invariant under generically finite map; see \cite{DS05} and \cite{Dan19} for details.

\item Let $f:X\to X$ be a surjective endomorphism of a smooth projective variety.
Denote by $$\chi_i(f):=\rho(f^*|_{H_{\textup{\'et}}^i(X,\mathbb{Q}_{\ell})})$$ 
where $\rho$ is the spectral radius, $H_{\textup{\'et}}^i(X,\mathbb{Q}_{\ell})$ is the $\ell$-adic cohomology, and $\ell$ is a prime different from the characteristic of the ground field $k$. When $X$ is an abelian variety, a characteristic free proof by Hu (cf.~\cite[Theorem 1.2 and Formula (3.6)]{Hu19}) showed that $$\chi_{2}(f)=d_1(f)=\chi_{1}(f)^2.$$

\item A normal (not necessarily projective) variety $X$ is a \textit{toric variety} if $X$ contains an algebraic torus $T=(k^*)^n$ as an (affine) open dense subset such that the natural multiplication action of $T$ on itself extends to an action on the whole variety. 

\item A commutative algebraic group $S$ is {\it semi-abelian} if $S$ is an extension of an abelian variety $A$ by an algebraic torus $T$.

\item Let $f:X\to X$ be a dominant self-morphism of a variety, which is not necessarily projective.  
A closed subvariety $Y\subseteq X$ is called \textit{$f$-invariant} (resp.\,\textit{$f$-periodic}) if $\overline{f(Y)}=Y$ (resp.\,$\overline{f^{s}(Y)}=Y$ for some $s>0$).
An $f$-periodic closed subvariety $V$ is said to be of \textit{Small Dynamical Degree} 
({\it SDD} for short) if the first dynamical degree $d_1(f^s|_V)<d_1(f^s)$ for some $s>0$ such that $\overline{f^s(V)}=V$.

\item Denote by $\Sigma_{f^{\infty}}$ the set of $f$-periodic (irreducible closed) subvarieties of Small Dynamical Degree ({\it SDD} for short).
We define the following subsets of $\Sigma_{f^{\infty}}$:

\begin{longtable}{p{1cm} p{12cm}}
$S_{f^{\infty}}$ & the subset of maximal $f$-periodic subvarieties in $\Sigma_{f^{\infty}}$,\\
$S_f$    &the subset of $S_{f^{\infty}}$,
consisting of all $f$-invariant elements,\\
$P_f$ &the subset of $S_{f^{\infty}}$,
consisting of all $f$-invariant prime divisors,\\
$P_{f^{\infty}}$ &the subset of (maximal) $f$-periodic prime divisors in $\Sigma_{f^{\infty}}$,\\
$P^{-1}_{f^{\infty}}$ & the subset of (maximal) $f^{-1}$-periodic prime divisors in $\Sigma_{f^{\infty}}$.
\end{longtable}
\item An element $D\in S_f$ is called an \textit{$f$-invariant maximal $f$-periodic SDD subvariety}.
 \item Denote by $\Fix(f)$ (resp.\,$\Per(f)$) the set of fixed (resp.\,periodic) points of $f$.
\end{itemize}
\end{notation}

The following lemma is basic, and is used in the proof of Proposition \ref{lem_reduction_nef}.

\begin{lemma}\label{lem_nef_descend}
Let $\sigma:\widetilde{X}\to X$ be a birational morphism between normal projective varieties.
Let $L$ be a nef $\mathbb{R}$-Cartier divisor on $X$ and denote by $\widetilde{L}:=\sigma^*L$. 
Let $\widetilde{\pi}:\widetilde{X}\to Y$ be a morphism to another projective variety $Y$ such that  $\widetilde{L}\equiv \widetilde{\pi}^*H$ with $H$ an ample divisor on $Y$. Then $\widetilde{\pi}$ factors through a well-defined morphism $\pi:X\to Y$ with $L\equiv \pi^*H$. 
\end{lemma}

\begin{proof}
By Zariski's main theorem, $\sigma_*\mathcal{O}_{\widetilde{X}}=\mathcal{O}_X$. For every curve $C$ lying in a fibre of $\sigma$, it follows from the projection formula that 
$H\cdot \widetilde{\pi}_*C=\widetilde{\pi}^*H\cdot C=\sigma^*L\cdot C=0$.
Since $H$ is ample on $Y$, $\widetilde{\pi}(C)$ is a point and hence every fibre of $\sigma$ will also be contracted by $\widetilde{\pi}$. By the rigidity lemma (cf.~\cite[Lemma 1.15]{Deb01}), $\widetilde{\pi}$ factors through a well-defined morphism $\pi:X\to Y$. So $\pi^*H=\sigma_*\sigma^*\pi^*H\equiv\sigma_*\widetilde{L}= L$. 
\end{proof}

Given an $f$-equivariant surjective morphism $\pi:X\to Y$ with $g:=f|_Y$, it is natural for us to compare those $\sharp S_f$, $\sharp P^{-1}_{f^{\infty}}$, $\sharp P_f$ on $f$ with the corresponding ones on $g$.

When $\pi$ is finite, we formulate the following lemma to bound $\sharp S_g$ in terms of $\sharp S_{f^{d!}}$ with $d=\deg \pi$. It is applied to prove the finiteness of $\sharp S_g$ for the $Q$-abelian case in Corollary \ref{Q_abelian_case}. Our upper bound for $\sharp S_g$ in the proof of Lemma \ref{lem_invariant_finite_iff} is  not optimal at all. 

\begin{lemma}\label{lem_invariant_finite_iff}
Let $\pi:X\to Y$ be a finite surjective morphism of  degree $d$ between projective varieties. Let $f:X\to X$ and $g:Y\to Y$ be two surjective endomorphisms such that $\pi\circ f=g\circ \pi$. If $\sharp S_{f^{d!}}<\infty$, then $\sharp S_g<\infty$. Conversely, if $\sharp S_g<\infty$, then $\sharp S_{f}<\infty$.
\end{lemma}

\begin{proof}
Since $\pi: X \to Y$ is finite and $d_1(f)$ is invariant under finite map, $\pi$ induces a surjective map $\pi: S_{f^{\infty}} \to S_{g^{\infty}}$
while $\pi^{-1}$ induces a one-to-many surjective map $\pi^{-1}: S_{g^{\infty}} \to S_{f^{\infty}}$.
Then the lemma follows.
\end{proof}

When $\pi$ is a birational morphism, we can only consider prime divisors rather than higher codimensional subvarieties due to the complexity of  periodic subvarieties contained in the exceptional locus of $\pi$.

\begin{lemma}\label{bir_invariant_case}
Let $\pi: X\dashrightarrow Y$ be a birational map between projective varieties. 
Let $f:X\to X$ and $g:Y\to Y$ be two surjective endomorphisms such that $\pi\circ f=g\circ \pi$. 
Then $ \sharp P_{f^n}=\sharp P_{g^n} +O(1)$ as $n\to \infty$, and $\sharp P_{f^\infty}$ is finite if and only if so is $\sharp P_{g^\infty}$.
\end{lemma}

\begin{proof}
By taking the graph $\Gamma\subseteq X\times Y$, our $f$ lifts to a surjective endomorphism $(f\times g)|_{\Gamma}$.
So it suffices for us to consider the case when $\pi$ is a birational morphism.
Let $B$ be the closed subset of $Y$, over which $\pi$ is not isomorphic. Let $\mathcal{B}_1$ be the set of prime divisors in $B$ and $\mathcal{B}_2$ the set of prime divisors in $\pi^{-1}(B)$. 
For each $n>0$, we have a natural map $$\varphi_n:P_{f^n}\backslash \mathcal{B}_2\to P_{g^n}\backslash \mathcal{B}_1,$$ sending $D$ to $\pi(D)$. Note that $\varphi_n$ admits an inverse, sending $D'\in P_{g^n}\backslash{\mathcal{B}_1}$ to its proper transform $\pi_*^{-1}(D')$. 
In particular, setting $K:=\sharp \mathcal{B}_1+\sharp \mathcal{B}_2$, we have: 
$$\sharp P_{g^n}-K \le\sharp P_{f^n}\le \sharp P_{g^n}+K.$$
\end{proof}

By a similar proof, we have the following lemma.

\begin{lemma}\label{bir_totally_periodic_case}
Let $\pi: X\dashrightarrow Y$ be a birational map between projective varieties. Let $f:X\to X$ and $g:Y\to Y$ be two surjective endomorphisms such that $\pi\circ f=g\circ \pi$. Then $\sharp P^{-1}_{f^\infty}$ is finite if and only if so is $\sharp P^{-1}_{g^\infty}$.
\end{lemma}

\begin{proposition}\label{prop_fixed_points}
Let $f:X\to X$ be a surjective endomorphism of a projective variety. Then there are only finitely many points (as zero-dimensional subvarieties) in the set $S_f$.
\end{proposition}

\begin{proof}
We may assume $d_1(f)>1$.
Let $Y \subseteq S_f$ be the subset of $f$-fixed closed points. Suppose the contrary that $\sharp Y$ is infinite. Let $Z$ be an irreducible component of the closure of $Y$ in $X$ with $\dim (Z)>0$. Then $f|_Z=\textup{id}_Z$ and $1=d_1(f|_Z)<d_1(f)$, a contradiction to the maximality of the points in $Y$.
\end{proof}

\section{Connection of two arithmetic dynamical conjectures}\label{S:rel_2_conj}

In this section, we work over a number field $K$.
We will give the relation between two conjectures
in Theorem \ref{thm-equiv}, which also gives an arithmetic motivation of this paper.
We refer to \cite[Section 2]{MMSZ20} for the notation involved in this section.

The following are the two main conjectures studied in \cite{MMSZ20}. 

\begin{conjecture}\label{conj_ks} ({\bf Kawaguchi-Silverman Conjecture}, ~\cite{KS16})
Let $f: X \to X$ be a surjective endomorphism of a projective variety $X$ defined over a number field $K$.
Let $x \in X(\overline K)$ such that its arithmetic degree $\alpha_{f}(x)< d_1(f)$.
Then the orbit $O_f(x)$ is not Zariski dense in
$X_{\overline{K}} := X \times_K \overline{K}$.
\end{conjecture}

\begin{conjecture}\label{conj_zf} ({\bf small Arithmetic Non-Density Conjecture}, ~\cite[Conjecture 1.4]{MMSZ20})
Let $f:X\to X$ be a surjective endomorphism of a projective variety $X$ defined over a number field $K$.
Then the set $Z_f(d):=\{x\in X(d)\,|\, \alpha_{f}(x)< d_1(f)\}$ is not Zariski dense in $X_{\overline{K}} = X \times_K \overline{K}$
for any positive constant $d>0$.
\end{conjecture}

We copy \cite[Question 8.4]{MMSZ20} in the following.
Note that for the equivalence of the above two conjectures, we may only focus on $Y\in S_{f^{\infty}}$ in Question \ref{ques_subvar}\,(1) and (2); see the proof of Theorem \ref{thm-equiv}.

\begin{question}\label{ques_subvar} (\cite[Question 8.4]{MMSZ20})
Let $f: X \to X$ be a surjective endomorphism on a projective variety with $d_1(f)>1$.
Fix  $d>0$.

\begin{itemize}
\item[(1)]
Let $Y$ be an $f$-invariant subvariety of $X_{\overline{K}}$.
Then, is $X(d) \cap (f^{-s}(Y) \setminus f^{-(s-1)}(Y))$  empty for $s \gg 1$?

\item[(2)]
Is there a positive integer $N$ such that $Y(d)$ is empty for any $f$-periodic  subvariety $Y \subseteq X_{\overline K}$ of small dynamical degree whose period is larger than $N$?

\item[(3)] (cf.~Question \ref{Q1}\,(1))
Is $S_{f^s|_{X_{\overline{K}}}}$ a finite set for any $s \ge 1$?
\end{itemize}
\end{question}

Note that Conjecture \ref{conj_zf} implies Conjecture \ref{conj_ks}. Conversely, we have:

\begin{theorem}\label{thm-equiv}
Suppose Question \ref{ques_subvar} and Conjecture \ref{conj_ks} hold true.
Then Conjecture \ref{conj_zf} holds true.
\end{theorem}

\begin{proof}
Let $x\in Z_f(d)$, i.e., $\alpha_{f}(x)<d_1(f)$ and $x\in X(d)$.
By \cite[Lemma 2.8]{MMSZ20}, $$\overline{O_f(x)}=\{x\}\cup \{f(x)\} \cup \cdots \cup \{f^{t-1}(x)\}\cup \bigcup_{i=0}^{s-1}f^i(\overline{O_{f^s}(f^t(x))}),$$
which is a union of irreducible closed subsets and $Y_0:=\overline{O_{f^s}(f^t(x))}$ is $f^s$-invariant.
Conjecture \ref{conj_ks} implies that $\alpha_{f^s}(f^t(x))=d_1(f^s|_{Y_0})$.
Note that $$\alpha_{f^s}(f^t(x))=\alpha_{f}(x)^s<d_1(f)^s.$$
Then $Y_0\in \Sigma_{f^{\infty}}$ and $Y_0\subseteq Y$ for some $Y\in S_{f^{s'}}$ with $s'>0$.
Note that $f^t(x)\in Y\cap X(d)$.
By Question \ref{ques_subvar}\,(2), there exists a constant $N>0$ such that $Y\in S_{f^{N}}$ for any $Y$ satisfying $Y\cap X(d)\neq\emptyset$.
Question \ref{ques_subvar}\,(3) implies that $S_{f^{N}}$ is finite.
Together with Question \ref{ques_subvar}\,(1), there exists another constant $M>0$ such that $x\in \bigcup_{Y\in S_{f^{N}},\, 0\le i\le M} f^{-i}(Y)$.
Therefore $$Z_f(d)\subseteq \bigcup_{Y\in S_{f^{N}},\, 0\le i\le M} f^{-i}(Y)$$ which is not Zariski dense in $X_{\overline{K}}$.
\end{proof}

\section{Isogenies of algebraic groups}

We use the following notation throughout this section.

\begin{notation}\label{notation_proof}
Let $G$ be a connected algebraic group with the identity element $e$.

We focus on the isogeny (i.e., surjective group homomorphism with finite kernel)
$$g:G\to G$$ 
with $d_1(g)>1$, and
$$f:=R_a\circ g$$
where $R_a$ is the right multiplication map by some $a\in G$.
Note that there exists a $G$-equivariant embedding $G\xhookrightarrow{} \widetilde{G}$ to a (smooth) projective variety $\widetilde{G}$ (cf.~\cite[Theorem 2]{Bri10}).
Then the same proof for \cite[Lemma 5.6]{MS18} shows that 
$$d_1(f)=d_1(g)>1.$$

We recall that $g$ and hence $f$ are finite surjective endomorphisms.
So $f$-invariant (resp.\,$f$-periodic) closed subvarieties defined in Notation \ref{not-2.1} are in the usual sense without taking closure.
\end{notation}

\begin{lemma}\label{lem-translation-d1}
Let $H$ be a (connected closed) algebraic subgroup of $G$.
Suppose $f(Hb)=Hb$ for some $b\in G$.
Then $g(H) = H$ and $d_1(f|_{Hb})=d_1(g|_H)$.
\end{lemma}

\begin{proof}
Let $p:=f(b)b^{-1}\in H$.
We have
$$g(H)=g(Hb)g(b)^{-1}=f(Hb)a^{-1}(f(b)a^{-1})^{-1}=Hbf(b)^{-1}=Hp^{-1}=H.$$
Moreover, We have the following commutative diagram:
\[
\xymatrix@C=60pt@R=30pt{
H \ar[r]^{(R_p\circ g)|_{H}}\ar[d]_{R_b} & H\ar[d]_{R_{b}}\\
Hb \ar[r]^{f|_{Hb}}  & Hb 
}
\]
where $R_b:H\to Hb$ is an isomorphism of varieties.
So $d_1(f|_{Hb})=d_1((R_p\circ g)|_{H}) =d_1(g|_{H})$ (cf.~\cite[Lemma 5.6]{MS18}).
\end{proof}

The following result is a generalization of Lemma \ref{lem-translation-d1} to arbitrary closed subvariety of $G$, which plays a crucial role in many later results.
\begin{proposition}\label{prop-gen-d1}
Let $X$ be an $f$-invariant (irreducible closed) subvariety of $G$.
Fix some $b\in X$ and denote by $X_e:=Xb^{-1}$.
Let $H$ be the smallest algebraic subgroup of $G$ containing $X_e$.
Then $g(H)=H$, $f(Hb)=Hb$, and $d_1(g|_H)=d_1(f|_{Hb})\le d_1(f|_X)$.
\end{proposition}
\begin{proof}
Let $p:=f(b)b^{-1} \in X_e\subseteq H$. We have
$$g(H)\supseteq g(X_e)=g(X)g(b)^{-1}=f(X)f(b)^{-1}=X_ep^{-1}.$$
Then $X_e\subseteq X_ep^{-1}(X_ep^{-1})^{-1}\subseteq g(H)$ and hence $H\le g(H)$ by the construction of $H$.
Since $g$ is an isogeny, $g(H)=H$.

Consider the alternate multiplication map $s_n: X^{\times(2n)} \to H$ defined by 
\begin{align*}
s_n(x_1b, \ldots,x_{2n}b)&=(x_1b)(x_2b)^{-1}\cdots (x_{2i+1}b)(x_{2i+2}b)^{-1}\cdots (x_{2n-1}b)(x_{2n}b)^{-1}\\
&=x_1x_2^{-1}\cdots x_{2i+1}x_{2i+2}^{-1}\cdots x_{2n-1}x_{2n}^{-1}\\
&\in (X_eX_e^{-1})^{\cdot n}
\end{align*}
Then the image of $s_n$ is $(X_eX_e^{-1})^{\cdot n}=H$ when $n\gg 1$ by \cite[Proposition 2.2.6]{Spr09} and noting that $e\in X_e$.
Consider the following diagram with $n\gg 1$:
\begin{equation}\label{diagram1}\tag{*}
\xymatrix@C=50pt@R=30pt{
X^{\times(2n)} \ar[r]^{(f|_X)^{\times(2n)}} \ar@{->>}[d]_{s_n} & X^{\times(2n)} \ar@{->>}[d]^{s_n}\\
H \ar[r]^{g|_H} & H 
}
\end{equation}
Note that $f(x_ib)=g(x_i)pb$.
Then 
\begin{align*}
s_n(f(x_1b),\ldots,f(x_{2n}b))&=s_n(g(x_1)pb,\ldots, g(x_{2n})pb)\\
&=g(x_1)g(x_2)^{-1}\cdots g(x_{2n-1})g(x_{2n})^{-1}\\
&=(g\circ s_n)(x_1b,\ldots,x_{2n}b).
\end{align*}
So the diagram commutes.
By the product formula, we have
$$d_1(g|_H)\le d_1((f|_X)^{\times(2n)})=d_1(f|_X).$$

On the other hand, note that $$f(Hb)=g(Hb)a=Hg(b)a=Hf(b)b^{-1}b=(Hp)b=Hb.$$
By Lemma \ref{lem-translation-d1}, we have $d_1(f|_{Hb})=d_1(g|_H)\le d_1(f|_X)$.
\end{proof}

\begin{lemma}\label{lem-unique}
There exists a unique  $X\in S_{g^{\infty}}$ 
containing $e \in G$.
Moreover, $X$ is an algebraic subgroup
and $g$-invariant.
\end{lemma}

\begin{proof}
Take $X \in S_{g^\infty}$ containing $e$. 
It exists because $d_1(g)>1$ and $d_1(g|_{\{e\}})=1$. 
Take $s>0$ such that  $g^s(X)=X$. 
Take the smallest (connected closed) algebraic subgroup $H$ of $G$ containing $X$.
By Proposition \ref{prop-gen-d1}, $g^s(H)=H$ and $d_1(g^s|_H)\le d_1(g^s|_X)<d_1(g^s)$.
By the maximality of $X$, we have $X=H$ and it is an algebraic subgroup.

Take another $Y \in S_{g^\infty}$ containing $e$ and then take $t>0$ such that $g^t(Y)=Y$.
The argument in the previous paragraph shows that $Y$ is a (connected closed) algebraic subgroup.
Let $Z$ be the smallest algebraic group containing the $g^{st}$-invariant irreducible closed subvariety $XY$.
Apply the similar commutative diagram (\ref{diagram1}) and argument in the proof of Proposition \ref{prop-gen-d1},
\[
\xymatrix@C=50pt@R=30pt{
(X\times Y)^{\times(2n)} \ar[rr]^{(g^{st}|_X\times g^{st}|_{Y})^{\times(2n)}} \ar@{->>}[d]_{s_n} && (X\times Y)^{\times(2n)} \ar@{->>}[d]^{s_n}\\
Z \ar[rr]^{g^{st}|_{Z}} && Z 
}
\]
where $s_n:(X\times Y)^{\times(2n)} \to Z=(XY)(XY)^{-1}\cdots(XY)(XY)^{-1}$  defined by
$$s_n(x_1,y_1, \ldots,x_{2n},y_{2n})=(x_1y_1)(x_2y_2)^{-1}\cdots (x_{2n-1}y_{2n-1})(x_{2n}y_{2n})^{-1}$$
is surjective with $n\gg 1$.
Then we have $d_1(g^{st}|_Z)<d_1(g^{st})$ and hence $Z=X=Y$ by the maximality of $X$ and $Y$.
This shows the uniqueness of $X$.

Next we show that $X$ is $g$-invariant. Note that $g(X)$ is also $g^s$-invariant and contains $e$.
By the product formula, the following diagram shows $d_1(g^s|_{g(X)})\le d_1(g^s|_{X})<d_1(g^s)$:
\[
\xymatrix@C=50pt@R=30pt{
X \ar[r]^{g^s|_{X}} \ar[d]_{g} & X \ar[d]^{g}\\
g(X) \ar[r]^{g^s|_{g(X)}} & g(X) }
\]
Therefore, we have $g(X) \subseteq X$ by the maximality and uniqueness of $X$.
In particular, $X$ is $g$-invariant since $g$ is an isogeny.
This finishes our proof.
\end{proof}

\begin{definition}\label{def-Gf}
With the notation in Lemma \ref{lem-unique}, we denote by $$G(f)=G(g):=X$$ the \textit{unique} 
element of $S_{g^\infty}$ containing $e \in G$. 
Note that $G(f)$ is $g$-invariant and $G(f^m)=G(g^m)=G(f)$ for any $m>0$.

{\it We shall prove in the next section that $G(f)$ is actually a normal subgroup of $G$.}
\end{definition}

\begin{proposition}\label{prop-gp-maxsdd}
A subvariety $X \subseteq G$ is a member of $S_f$ if and only if 
$X=G(f)b$ with $f(b)b^{-1} \in G(f)$ ($= G(g)$).
\end{proposition}

\begin{proof}
Assume that $X \in S_f$.
Fix some $b\in X$ and denote by $X_e:=Xb^{-1}$ which contains $e\in G$.
Let $H$ be the smallest algebraic subgroup of $G$ containing $X_e$.
By Proposition \ref{prop-gen-d1}, $g(H)=H$, $f(Hb)=Hb$, and $d_1(g|_H)\le d_1(f|_X)<d_1(f)$.
Then $f(b)b^{-1}\in X_e\subseteq H\le G(g)$ by Lemma \ref{lem-unique}.
Note that 
$$f(G(f)b)=g(G(g))g(b)a=(G(g)f(b)b^{-1})b=G(f)b.$$
By Lemma \ref{lem-translation-d1}, 
$$d_1(f|_{G(f)b})=d_1(g|_{G(g)})<d_1(g)=d_1(f).$$
Since $X\subseteq Hb\subseteq G(f)b$, we have $X=G(f)b$ by the maximality of $X$.

Conversely,
assume that $X=G(f)b$ with $f(b)b^{-1} \in G(f)$.
Then $X \in \Sigma_f$ by the same argument above.
Suppose that we have $Y \in S_{f^\infty}$ containing $X$.
Take $s>0$ such that $Y \in S_{f^s}$.
Then $Y=G(f^s)c=G(f)c$ for some $c \in G$ with $f^s(c)c^{-1} \in G(f)$ by the 
previous paragraph.
But this implies that $X=Y$ since both are translations of the same group $G(f)$.
So $X$ is maximal.
\end{proof}

The following modified product formula will be frequently used in the next Theorem and also next section.

\begin{lemma}\label{lem-sf-induction}
Consider the commutative diagram of homomorphisms of connected algebraic groups:
\[
\xymatrix{
1 \ar[r] & B\ar[d]^{g|_B} \ar@{^{(}->}[r] & A \ar[r]^{\pi}\ar[d]^g & C\ar[r]\ar[d]^{g|_{C}}& 1\\
1 \ar[r] & B \ar@{^{(}->}[r] & A \ar[r]^{\pi}& C \ar[r]& 1
}
\]
where $g$ is an isogeny. Then $\pi$ is $f$-equivariant and the following hold.
\begin{enumerate}
\item  $d_1(g|_{\pi})=d_1(g|_B)$. In particular, $d_1(g)=\max\{d_1(g|_B), d_1(g|_C)\}$.
\item  Suppose $d_1(g|_B)<d_1(g)$. Then $A(f)=\pi^{-1}(C(f|_C))$ and $\pi$ induces a bijection between $S_f$ and $S_{f|_C}$.
\end{enumerate}
\end{lemma}

\begin{proof}
First, note that $\pi$ is $f$-equivariant with $f|_C=R_{\pi(a)}\circ g|_C$.

Let $A$ act on $A$ and $C$ by the right multiplication. 
By \cite[Theorem 2]{Bri10}, we have $A$-equivariant open embeddings to smooth projective varieties $A\hookrightarrow \overline{A}$ and $C\hookrightarrow \overline{C}$
(taking $A$-equivariant resolution if necessary).
Denote by $\overline{\pi}:\overline{A}\dashrightarrow \overline{C}$ the induced dominant rational map.
Denote by $\overline{g}:\overline{A}\dashrightarrow \overline{A}$ and $\overline{g}|_{\overline{C}}:\overline{C}\dashrightarrow \overline{C}$ the induced dominant self-maps.
Note that the graph of $\overline{\pi}$ is $A$-equivariant.
Replacing $\overline{A}$ by the $A$-equivariant resolution of this graph,
we may further assume that $\overline{\pi}$ is a well-defined algebraic fibration.

By the generic smoothness of $\overline{\pi}$, there exists a general point $c_1=\pi(a_1)\in C$ such that the fibre $\overline{A}_{c_1}:=\overline{\pi}^{-1}(c_1)$ is an (irreducible) smooth projective variety. 
For any $c_2=\pi(a_2)\in C$, we have $\overline{A}_{c_2}=\overline{A}_{c_1\cdot (c_1^{-1}c_2)}=\overline{A}_{c_1}\cdot (a_1^{-1}a_2)$.
So all the fibres of $\overline{\pi}$ over $C$ are (irreducible) smooth projective varieties isomorphic to each other.

We refer to \cite{DN11} for the definitions of critical value and critical set.
Now we have that each point of $C$ is not a critical value of $\overline{\pi}$. 
Note that $B$ is $g$-invariant and $g$ is \'etale.
Then $\overline{A}_e$ (the closure of $B$ in $\overline{A}$) is not contained in the critical set of $\overline{g}$.
Therefore, we verified the assumptions for $\overline{A}_e$ in \cite[Remark 3.4]{DN11} and hence $d_1(\overline{g}|_{\overline{\pi}})=d_1(\overline{g}|_{\overline{A}_e})=d_1(g|_B)$.
Together with the product formula, (1) is proved.

For (2), $d_1(g)=d_1(g|_C)$ by (1).

We apply the maximality and uniqueness of $A(g)$ and $C(g|_C)$ (cf.~Lemma \ref{lem-unique}) and our new product formula (1) several times in this diagram.
First, we have $B\le A(g)$.
Note that $d_1(g|_{\pi(A(g))})=d_1(g|_{A(g)})<d_1(g)$.
So we have $\pi(A(g))\le C(g|_C)$.
Conversely, $d_1(g|_{\pi^{-1}(C(g|_C))})=d_1((g|_C)|_{C(g|_C)})< d_1(g|_C) = d_1(g)$.
So we have $A(g)=\pi^{-1}(C(g|_C))$.

Finally, we apply Proposition \ref{prop-gp-maxsdd} for $f$ and $f|_C$,
then we see that the induced map 
$$\pi: S_f\to S_{f|_C}, \hskip 1pc \text{ via } \hskip 1pc A(f)b\mapsto C(f|_C)\pi(b)$$  is a bijection.
So (2) is proved.
\end{proof}

We complete this section by considering the commutative algebraic groups first.

\begin{theorem}\label{thm-gp}
Let $f:G\to G$ be a dominant self-morphism of a connected \textbf{commutative} algebraic group $G$, such that $f=R_a\circ g$ where $g$ is an isogeny and $a\in G$.
Suppose $d_1(f)>1$. 
Let $G(f)$ be as in Definition \ref{def-Gf}. Then
$$\pi:G\to \overline{G}:=G/G(f)$$ 
is $f,g$-equivariant and
$$0<\sharp S_f=\sharp S_g=\sharp  \Fix(\overline{f})=\sharp  \Fix(\overline{g})<\infty$$
where $\overline{f}=f|_{\overline{G}}=R_{\pi(a)}\circ\overline{g}$.
\end{theorem}

\begin{proof}
We simply define $\overline{f}(x\cdot G(f)):=f(x)\cdot G(f)$.
It is well-defined and $\overline{f}\circ\pi=\pi\circ f$ since $g(G(f))=G(f)$.
By Lemma \ref{lem-sf-induction} and
Proposition \ref{prop-gp-maxsdd}, 
$S_{\overline f} = \Fix(\overline{f})$ and
$\pi$ induces a bijection between $S_f$ and $\Fix(\overline{f})$.

Consider the case $\sharp \Fix(\overline{g})=\infty$.
Then there is a positive dimensional (connected closed) algebraic subgroup $\overline{H}$ of $\overline{G}$ such that $\overline{g}|_{\overline{H}}=\id_{\overline{H}}$.
Let $H:=\pi^{-1}(\overline{H})$ which is a $g$-invariant (connected closed) algebraic subgroup of $G$.
By the product formula in Lemma \ref{lem-sf-induction}, we have
$$d_1(g|_H)=\max\{d_1(g|_{\pi}), d_1(\overline{g}|_{\overline{H}})\}=d_1(g|_{\pi})=d_1(g|_{G(g)})<d_1(g).$$
This contradicts the maximality of $G(g)$.

Now we have shown $\sharp\Fix(\overline{g})<\infty$.
Then $\phi : \overline{G} \to \overline{G}$ ($\overline{x} \mapsto g({\overline x}) \cdot {\overline x}^{-1}$)
is an isogeny.
Note that $\overline f(\overline x)=\overline x$ if and only if 
$\phi(\overline x)=\overline{a^{-1}}$.
So we have $$0< \sharp\Fix(\overline{f}) = 
\sharp \,\phi^{-1}(\overline{a^{-1}}) =
\sharp \,\phi^{-1}(\overline{e}) = 
\sharp \Fix(\overline{g})
<\infty$$ 
and the theorem is proved.
\end{proof}

\section{Proof of Theorem \ref{main-thm-gp}}\label{sec-5}

In this section, we focus on the upper bound in Question \ref{Q2}.
We first consider the case of abelian varieties.
\begin{theorem}\label{thm-ab}
Let  $f$ be a surjective endomorphism of an abelian variety $A$.
Then
$$\sharp S_f\le (\sqrt{d_1(f)}+1)^{2(\dim (A)-\dim (A(f)))}$$
(cf.~Definition \ref{def-Gf} for $A(f)$).
\end{theorem}

\begin{proof}
We may assume $d_1(f)>1$.
Write $f=R_a\circ g$ where $g$ is an isogeny and $a\in A$.
Since $d_1(f) = d_1(g)$ and $A(f) = A(g)$,
applying Theorem \ref{thm-gp}, we may assume that $f$ is an isogeny, i.e., $f = g$; further we have an $f$-equivariant quotient morphism
$$\pi:A\to A':=A/A(f)$$
such that $0<\sharp S_f=\sharp\Fix(\overline{f})<\infty$
where $f'=f|_{A'}$.
Denote by $\phi:A'\to A'~(x'\mapsto f'(x')\cdot x'^{-1})$ (see the ending part of the proof of Theorem \ref{thm-gp}), which is an isogeny.
Then $$\sharp\Fix(f')=\deg(\phi).$$ 
Set $m:=\dim (A')=\dim(A)-\dim(A(f))$.
We compute 
\begin{align*}
\deg(\phi) &= \chi_{2m}(\phi) \\
&= \rho (\phi^*|_{H^{2m}_{\textup{\'et}}(A',\mathbb{Q}_{\ell})}) \\
&= \rho \left((\wedge^{2m}\phi^*)|_{\wedge^{2m}H_{\textup{\'et}}^{1}(A',\mathbb{Q}_{\ell})}\right) \\
&= | \det(\phi^*|_{H_{\textup{\'et}}^1(A',\mathbb{Q}_{\ell})}) |.
\end{align*}
Note that $\phi^*|_{H^1_{\textup{\'et}}(A',\mathbb{Q}_{\ell})}=f'^*|_{H^1_{\textup{\'et}}(A',\mathbb{Q}_{\ell})}-\id_{H^1_{\textup{\'et}}(A',\mathbb{Q}_{\ell})}$ 
(cf.~\cite[Theorem 2.2]{Hu19}).
So, letting $\alpha_1,\alpha_2,\ldots,\alpha_{2m}$ be the eigenvalues of 
$f'^*|_{H^1_{\textup{\'et}}(A',\Q_{\ell})}$, 
we have 
$$\deg(\phi)=\prod_{i=1}^{2m}|\alpha_i - 1| \leq (\chi_1(f')+1)^{2m}.$$
Note that $\chi_1(f')^2=d_1(f')\le d_1(f)$ by \cite[Formula (3.6)]{Hu19} and the product formula. 
Then the theorem follows.
\end{proof}

Recall that a normal projective variety $X$ is said to be \textit{$Q$-abelian} (cf.~\cite{NZ10}) if there is a finite surjective morphism $A\to X$ from an abelian variety $A$, which is \'etale in codimension one.

\begin{corollary}\label{Q_abelian_case}
$S_f$ is finite for any surjective endomorphism $f$ of a $Q$-abelian variety.
\end{corollary}

\begin{proof}
By taking the Albanese closure $\pi:A\to X$ (cf.~\cite[Lemma 8.1]{CMZ20}), $f$ lifts to a surjective endomorphism $f_A:A\to A$ of an abelian variety $A$. Then our corollary follows from Lemma \ref{lem_invariant_finite_iff} and Theorem \ref{thm-ab}.
\end{proof}

Next we treat the semi-abelian varieties.
Let $f:X\to X$ be a (not necessarily dominant) self-morphism of a semi-abelian variety $X$.
Then $R_{f(e)^{-1}}\circ f$ is a group homomorphism of $X$ (cf.~\cite[Theorem 2]{Itk76}).

When $f$ is a group homomorphism,
it induces the following commutative diagram
where $T$ is an algebraic torus and $A = \Alb(X)$ is an abelian variety:
\[
\xymatrix{
0 \ar[r] & T\ar[d]^{f|_T} \ar@{^{(}->}[r] & X \ar[r]\ar[d]^f & A \ar[r]\ar[d]^{f|_A}& 0\\
0 \ar[r] & T \ar@{^{(}->}[r] & X \ar[r]& A \ar[r]& 0
}
\]

\begin{thm}\label{thm-semi}
Let $f$ be a dominant self-morphism of a semi-abelian variety $X$. 
Then
\[
\sharp S_f \leq ( \sqrt{d_{1}(f)}+1)^{2(\dim (X)-\dim(X(f)))}.
\]
\end{thm}

\begin{proof}
By the same argument as in the first paragraph of the proof of Theorem \ref{thm-ab}, 
we may replace $f$ by $g$ and $X$ by $X/X(f)$ so 
that $\dim(X(f))=0$ and both $f$ and $\phi:X\to X ~(x\mapsto f(x)\cdot x^{-1})$ are isogenies.

Consider the commutative diagram before Theorem \ref{thm-semi}. 
Note that $f|_T$ is defined by a matrix $M$ of size $\dim(T)$ with integer coefficients and a non-zero determinant.
Then $d_{1}(f|_T)$ is equal to the spectral radius of $M$ (cf.~\cite[Theorem 1.1]{Lin12}).
Let $\lambda_1,\cdots,\lambda_{\dim(T)}$ be the diagonal entries of the Jordan canonical form of $M$.

We compute
\begin{align*}
\sharp \Fix(f|_T) &= \deg (\phi|_T)= \det (M-\id) \\
&= ( \lambda_{1}-1)\cdots ( \lambda_{\dim (T)}-1)   \\
&\leq (d_{1}(f|_T)+1)^{\dim (T)}.
\end{align*}
Together with the computation in the proof of Theorem \ref{thm-ab} for abelian varieties, we have
\begin{align*}
\sharp\Fix(f) &= \sharp\Fix(f|_T)\cdot \sharp \Fix(f|_A) \\
&\le (d_{1}(f|_T)+1)^{\dim (T)} \cdot (\sqrt{d_1(f|_A)}+1)^{2\dim(A)} \\
&\le ( \sqrt{d_{1}(f)}+1)^{2\dim (X)}
\end{align*}
where the last inequality follows from $d_1(f|_T)\le d_1(f)$ and $d_1(f|_A)\le d_1(f)$; see e.g. \cite[Theorem 1.6]{MS18} or Lemma \ref{lem-sf-induction}. 
So the theorem is proved.
\end{proof}

\begin{remark}\label{rmk-semi}
In the proof of Theorem \ref{thm-semi}, we see that if $X=T$ is an algebraic torus, 
then the computation indeed shows that $\sharp S_f\le (d_1(f)+1)^{\dim(T)-\dim(T(f))}$.
\end{remark}

We now prepare to deal with general algebraic groups.
See, for instance, \cite[\S 3]{MZ18a} for several decomposition facts about algebraic groups.
First, we need observations for unipotent groups and semisimple groups below.

\begin{lemma}\label{lem-unipotent}
Let $g:U\to U$ be an isogeny of a (connected) unipotent group $U$.
Then $d_1(g)=1$.
\end{lemma}

\begin{proof}
Since $\Ker(g)$ is a finite subgroup of the unipotent
group $U$, it is trivial, so $g$ is an automorphism.
We show by induction on $\dim(U)$ and assume $\dim(U)>0$.
Note that $U$ is solvable.
So its commutator subgroup $U^{(1)}$ is strictly smaller than $U$.
Note that $g(U^{(1)})=U^{(1)}$.
Consider the following commutative diagram:
\[
\xymatrix{
1 \ar[r] & U^{(1)}\ar[d]^{g|_{U^{(1)}}} \ar@{^{(}->}[r] & U \ar[r]\ar[d]^g & U/U^{(1)} \ar[r]\ar[d]^{g|_{U/U^{(1)}}}& 1\\
1 \ar[r] & U^{(1)} \ar@{^{(}->}[r] & U \ar[r]& U/U^{(1)} \ar[r]& 1
}
\]

Suppose $U^{(1)}$ is not trivial. 
By Lemma \ref{lem-sf-induction} and induction, we have 
$$d_1(g)=\max\{d_1(g|_{U^{(1)}}), d_1(g|_{{U/U^{(1)}}})\}=1.$$

Suppose $U^{(1)}$ is trivial. 
Then $U$ is a commutative unipotent group which is isomorphic to $k^{\times {\dim(U)}}$.
One checks that $g\in \GL(U)$ since $g$ is additive.
By \cite[Theorem 2]{Bri10}, there is a $\GL(U)$-equivariant open embedding $U\xhookrightarrow{} X$ to some smooth projective variety $X$ (taking equivariant resolution if necessary).
Thus the connected group $\GL(U)$ and hence $g$ act trivially on the lattice $\NS(X)/{\rm (torsion)}$.
So $d_1(g)=1$.
\end{proof}

\begin{lemma}\label{lem-ss}
Let $g:S\to S$ be an isogeny of a (connected) semisimple group $S$.
Then $d_1(g)=1$.
\end{lemma}

\begin{proof}
First note that $(\deg g)^n=\deg g^n=\sharp\Ker(g^n)\le \sharp Z(S)<\infty$ for any $n>0$, where $Z(S)$ is the centre of $S$ (a finite set). Thus $g\in \Aut(S)$, the group of algebraic group automorphisms.
Let $\pi:\widetilde{S}\to S$ be the universal cover of $S$ which is finite.
Then $g$ lifts to $\widetilde{g}:\widetilde{S}\to \widetilde{S}$.
By the product formula, we may replace $S$ by $\widetilde{S}$ and assume $S$ is simply connected.
By the homomorphisms theorem and \cite[Proposition 2.2]{HM69}, $\Aut(S)\cong \Aut(\mathfrak{s})$ in the sense of algebraic groups via the differential map at $e$, where $\mathfrak{s}$ is the Lie algebra of $S$.
By \cite[Proposition D.40]{FH91}, $\Aut(S)/\Aut^{\circ}(S)\cong \Aut(\mathfrak{s})/\Aut^{\circ}(\mathfrak{s})$ is finite.
Replacing $g$ by a power, we may assume $g$ is in the connected algebraic group $\Aut^{\circ}(S)$.
We are done by the same argument at the ending part of Lemma \ref{lem-unipotent}. 
\end{proof}

Finally, we are able to give the same upper bound for any algebraic group.

\begin{theorem}\label{thm_alg_group}
Let $G$ be a connected algebraic group.
Then $G(f)$ is normal in $G$ and 
$$\sharp S_f\le (\sqrt{d_1(f)}+1)^{2(\dim (G)-\dim (G(f)))}$$
for any self-morphism $f=R_a\circ g$ where $g$ is an isogeny and $a\in G$.
\end{theorem}

\begin{proof}
We may assume $1 < d_1(f)$ ($= d_1(g)$).
Let $G_a$ be the maximal connected affine normal subgroup of $G$, 
and $U$ the unipotent radical of $G_a$.
Then $g(G_a)=G_a$ and $g(U)=U$.
Consider the following commutative diagram:
\[
\xymatrix{
1 \ar[r] & U\ar[d]^{g|_U} \ar@{^{(}->}[r] & G \ar[r]^{\pi}\ar[d]^g & G/U \ar[r]\ar[d]^{g|_{G/U}}& 1\\
1 \ar[r] & U \ar@{^{(}->}[r] & G \ar[r]^{\pi}& G/U \ar[r]& 1
}
\]
By Lemma \ref{lem-unipotent}, we have $d_1(g|_U)=1<d_1(g)$.
By Lemma \ref{lem-sf-induction}, 
$\pi$ is $f$-equivariant and $U\le G(f)=\pi^{-1}(G(f|_{G/U}))$.
Then $G(f)$ is normal in $G$ if and only if $(G/U)(f|_{G/U})$ is normal in $G/U$.
Therefore, we may replace $G$ by $G/U$ and assume that $U$ is trivial.

Let $S:=G^{(1)}$ be the commutator subgroup of $G$.
Then $g(S)=S$, and $S=(G_a)^{(1)}$ which is a semisimple group since $G_a$ is a reductive group now.
Consider the following commutative diagram:
\[
\xymatrix{
1 \ar[r] & S\ar[d]^{g|_S} \ar@{^{(}->}[r] & G \ar[r]^{\pi}\ar[d]^g & G/S \ar[r]\ar[d]^{g|_{G/S}}& 1\\
1 \ar[r] & S \ar@{^{(}->}[r] & G \ar[r]^{\pi}& G/S \ar[r]& 1
}
\]
By Lemma \ref{lem-ss},  we have $d_1(g|_S)=1$. 
By Lemma  \ref{lem-sf-induction} and an argument similar to the previous paragraph, we may replace $G$ by $G/S$ and assume that $G$ is commutative.

Note that now $G$ is a semi-abelian variety.
In particular, $G(f)$ is normal in $G$, and the upper bound is as claimed, by Theorem \ref{thm-semi} and Lemma \ref{lem-sf-induction}.
\end{proof}

At the end of this section, we treat the toric varieties.
Let $X$ be a toric variety and $T$ the big torus. A subvariety of the form $O=x\cdot T$ is called a {\it torus orbit} on $X$.
Note that $X=\coprod_{O\subseteq X} O$ is a finite disjoint union of all torus orbits on $X$.

\begin{theorem}\label{thm-toric}
Let $X$ be a toric variety with the big torus $T \subseteq X$.
Let $f \colon X \rightarrow X$ be a dominant self-morphism such that $f(T)\subseteq T$.
Then
\[
\sharp S_{f} \leq \sum_{O \subseteq X} (d_{1}(f)+1)^{\dim (O)}
\]
where the sum runs over all the torus orbits
on $X$.
\end{theorem}

\begin{proof}
By \cite[Lemma 3.4.3]{Nak08}, $f=R_a\circ g$ such that $g|_T$ is an isogeny  and $g(x\cdot t)=g(x)\cdot g(t)$; see~\cite[Section 3.4]{Nak08} for details.
Here $R_a$
is the natural extension of the right multiplication map on $T$ by $a\in T$.
Since $g|_T$ is dominant, $f(T)=g(T)\cdot a=T$.
Moreover, $f(x\cdot T)=g(x\cdot T)\cdot a=g(x)\cdot T$.
In particular, $f$ maps every torus orbit onto a torus orbit.
Note that $d_1(f|_T)=d_1(f)$ since it is a birational invariant.

Let $V \in S_{f}$. 
Recall that $\overline{f(V)}=V$ and $d_1(f|_V)<d_1(f)$.
Note that $V=\coprod_{O\subseteq X} V\cap O$.
Then there is a torus orbit $O$ such that $V\cap O$ contains an open dense subset of $V$.
So $f(O)\cap O\supseteq f(V\cap O)\cap (V\cap O)\neq\emptyset$ and hence $f(O)=O$.
Note that $d_1(f|_O)=d_1(g|_O)$ (cf.~\cite[Lemma 5.6]{MS18}). 

Now fixing any torus orbit $O$ with $f(O)=O$, we compute the number of $V\in S_f$ such that $V\cap O$ contains an open dense subset of $V$. 
Fix any point $x_0\in O$.
Write $f(x_0)=x_0\cdot b$ for some $b\in T$.
Define $f_O:O\to O$ via $f_O:=R_{b^{-1} \cdot a}\circ f|_O$.
Then $d_1(f_O)=d_1(f|_O)$ (cf.~\cite[Lemma 5.6]{MS18}).
Define the natural surjective morphism
$\pi: T\to O$ via $\pi(t)=x_0\cdot t$.
Note that 
$$f_O\circ \pi(t)=f(x_0\cdot t)\cdot b^{-1}\cdot a=g(x_0)\cdot g(t)\cdot a\cdot b^{-1}\cdot a=x_0\cdot f(t)=\pi\circ f|_T(t).$$
By the product formula, $d_1(f|_O)=d_1(f_O)\le d_1(f|_T)=d_1(f)$. 

If $d_{1}(f|_{O})<d_{1}(f)$, then $V= \overline{O}$ by the maximality.
If $d_{1}(f|_{O})\ge d_{1}(f)$, then $V\cap O\in S_{f|_O}$.
Applying Remark \ref{rmk-semi} to $f|_{O}$,
the number of such $V$ is bounded by $(d_{1}(f)+1)^{\dim (O)}$.
This finishes the proof.
\end{proof}

\begin{remark}
In Theorem \ref{thm-toric}, as we do not assume $X$ is projective, $f$ is not necessarily finite (nor surjective).
Also during its proof, it is possible that $O\subseteq \Sing(X)$, and then we cannot deduce $d_1(f|_O)\le d_1(f)$ from \cite[Lemma 5.3]{MS18} where we need to take (projective) resolution of $X$. 
Of course, if $X$ (and hence $\overline{O}$) are projective, then there is no issue here.
\end{remark}

\section{Reduction to polarized case: Proof of 
Theorem \ref{main-thm-all}}

In this section, the ground field $k$ is allowed to have arbitrary characteristic.
We will prove Theorem \ref{thm-all} which includes
Theorem~\ref{main-thm-all} as a special case.
The connection with polarized endomorphism will give us finiteness for $\sharp P_f$. 

\begin{proposition}\label{prop-pf-bound}
Let $\pi:X\dashrightarrow Y$ be a dominant rational map between projective varieties.
Let $f:X\to X$ and $g:Y\to Y$ be surjective endomorphisms such that $g\circ\pi=\pi\circ f$,  $\dim(Y)>0$, and $g$ is $d_1(f)$-polarized.
Then  for $n\gg 1$,
$$\sharp P_{f^n}\le (d_1(f))^n(1+o(1)) .$$ \end{proposition}

\begin{proof}
By Lemma \ref{bir_invariant_case}, replacing $\pi$ with its graph and taking the normalizations, we may assume that $\pi$ is a surjective morphism and both $X$ and $Y$ are normal.

Take the Stein factorization $X\to Y_1\to Y$ such that $p_1:X\to Y_1$ has connected fibres and $p_2: Y_1\to Y$ is finite. By \cite[Lemma 5.2]{CMZ20}, $f$ descends to a surjective endomorphism $h:Y_1\to Y_1$ such that $p_1\circ f=h\circ p_1$ and $p_2\circ h=g\circ p_2$. Our $h$, like $g$, is also $d_1(f)$-polarized. Hence we may assume further $\pi$ has connected fibres.

For any $D \in P_{f^\infty} =\bigcup_{n=1}^{\infty} P_{f^n}$, say $D \in P_{f^s}$,
$\pi(D)$ is a point; otherwise, $d_1(f^s)>d_1(f^s|_D)\ge d_1(g^s|_{\pi(D)})=d_1(g^s)=d_1(f^s)$ since both $g^s$ and $g^s|_{\pi(D)}$ are ($d_1(f)^s =$) $d_1(f^s)$-polarized, a contradiction. 
For the upper bound, we may assume $P_{f^\infty}$
is an infinite set whose elements are contained in fibres of $\pi$.
Then $\dim Y = 1$.
Let $\mathcal{B}$ be the finite set of prime divisors in the (finitely many) reducible fibres of $\pi$.
We have an injection 
$$\varphi:P_{f^n}\backslash \mathcal{B}\to \Fix(g^n),$$
sending $D$ to $\pi(D)$. 
Then our argument follows from Proposition \ref{lem-lefschetz}.
\end{proof}


The following result 
is well known when the ground field $k$ has $\ch \, k = 0$ and the proof is not that troublesome due to the Hodge decomposition.
Our proof here works when $\ch \, k$ is arbitrary, and is related to a generalized problem of Weil's Riemann Hypothesis. 

\begin{proposition}\label{lem-lefschetz}
Let $f:C\to C$ be a non-isomorphic endomorphism  on a smooth projective curve $C$. Then $\sharp\Fix(f)\le \deg f+2c(\deg f)^{1/2}+1$ where $c$ is the genus of $C$.	
\end{proposition}

\begin{proof}
Let $d:= d_1(f) = \deg f \ge 2$.
Then $f$ is $d$-polarized, so is its induced
endomorphism $g:\Alb(C)\to\Alb(C)$ by
Theorem \ref{thm-all}. Note that $f$ has at most finitely many fixed points.
Applying Lefschetz fixed-point formula, we have 
$$\sharp \textup{Fix}(f)\le\sum_{i=0}^2(-1)^i\textup{Tr}(f^*|_{H_{\textup{\'et}}^i(C,\mathbb{Q}_{\ell})}).$$
Identifying $H_{\textup{\'et}}^1(C,\mathbb{Q}_{\ell})$ with $H_{\textup{\'et}}^1(\Alb(C),\mathbb{Q}_{\ell})$, 
we have $f^*|_{H_{\textup{\'et}}^1(C,\mathbb{Q}_{\ell})}=g^*|_{H_{\textup{\'et}}^1(\Alb(C),\mathbb{Q}_{\ell})}$.
Since $\Alb(C)$ is an abelian variety, we have $\chi_1(f)^2=\chi_1(g)^2=d_1(g)=d$ (cf.~\cite[Formula (3.6)]{Hu19}).
So the proposition is proved. 
\end{proof}

We are going to prove Theorem \ref{thm-all} (cf.~Theorem \ref{main-thm-all}).
It has its own independent interest and various applications, and is also needed in proving Proposition \ref{lem-lefschetz} in positive characteristic.
It generalizes \cite[Theorem 1.2 and Theorem 5.1]{CMZ20} by removing the separable assumption, \cite[Proposition 3.3]{MZ19} by removing the surjective assumption, \cite[Lemma 2.3]{NZ10} without the polarized and characteristic 0 assumptions, and \cite[Remark 5.9]{San17} without the smooth and characteristic 0 assumptions.

We prepare some notation. Let $f$ be a surjective endomorphism of a normal projective variety $X$ over the field $k$ which is allowed to be of \textbf{arbitrary characteristic}.
Let
$$X^{\vee}:=\Pic^\circ(X)_{\red} ,$$
$$f^{\vee}:=f^*|_{\Pic^\circ(X)_{\red}}: X^{\vee}\to X^{\vee}$$
the dual of $X$ and dual isogeny of $f$.
Denote by
$$\Alb(X):=\Pic^\circ(\Pic^\circ(X)_{\red})=X^{\vee\vee}$$
the Albanese variety of $X$.
There are two canonical maps
$$\alb_X:X\to \Alb(X), \hskip 1pc
\mathfrak{alb}_X:X\dashrightarrow \mathfrak{Alb}(X)$$
called {\it Albanese morphism} and {\it Albanese map} satisfying the universal properties in the categories of morphisms and rational maps, respectively (cf.~\cite[Remark 9.5.25]{FGI+05}, \cite[Chapter II.3]{Lan83} and \cite[Section 5]{CMZ20}).
In particular, $f$ descends to surjective endomorphisms on $\Alb(X)$ and $\mathfrak{Alb}(X)$.

We need the following in proving Theorem \ref{thm-all}.

\begin{proposition}\label{prop-sqrt}
Let $f:X\to X$ be a (not necessarily separable) surjective endomorphism of a normal projective variety $X$ over an algebraically closed field $k$ of arbitrary characteristic.
Let $\lambda$ be an eigenvalue of $f^*|_{\Pic^\circ(X)_{\red}}$, equivalently of $(f^{\vee})^*|_{H^1(X^{\vee}, \mathbb{Z}_\ell)}$. 
Then 
$$\iota_f\le |\lambda|^2\le d_1(f)$$
where $\iota_f$ is the minimal modulus of eigenvalues of $f^*|_{\N^1(X)}$.
\end{proposition}

\begin{proof}
Let $P\in \mathbb{Z}[t]$ be the minimal polynomial of $f^{\vee}:X^{\vee}\to X^{\vee}$, which is the minimal polynomial of $(f^{\vee})^*|_{H^1(X^{\vee}, \mathbb{Z}_\ell)}$.
Then $P(f^{\vee})=0$ (cf.~\cite[Section 2]{Hu19}).
Note also that $P(f^{\vee\vee})=P(f^{\vee})^{\vee}=0$.
By a dual argument, $P$ is also the minimal polynomial of $f^{\vee\vee}$, i.e., the minimal polynomial of $(f^{\vee\vee})^*|_{H^1(X^{\vee\vee},\mathbb{Z}_l)}$.
Let $a$ and $b$ be the minimal and maximal modulus of roots of $P$, respectively.
By \cite[Formula (3.6)]{Hu19}, $b^2=d_1(f^{\vee})$.
The same proof for \cite[Formula (3.6)]{Hu19} shows that $a^2=\iota_{f^{\vee}}$.
A similar argument for $f^{\vee\vee}$ says $b^2=d_1(f^{\vee\vee})$ and $a^2=\iota_{f^{\vee\vee}}$. 
Moreover, note that $f|_{\Alb(X)}=f^{\vee\vee}+t$ for some $t\in X^{\vee\vee}=\Alb(X)$.
Then we have $\iota_{f|_{\Alb(X)}}=\iota_{f^{\vee\vee}}=\iota_{f^{\vee}}$ and $d_1(f|_{\Alb(X)})=d_1(f^{\vee\vee})=d_1(f^{\vee})$.
Let $\lambda$ be an eigenvalue of $f^{\vee}$.
Then $\lambda$ is a root of $P$ and hence $a\le |\lambda|\le b$.

We may assume $0\in \alb_X(X)$.
Note that the smallest abelian subvariety of $A$ containing $\alb_X(X)$ is $A$ itself by the universal property.
Applying the similar technique used in the proof of Proposition \ref{prop-gen-d1},
we have the following commutative diagram
\begin{equation}\label{diag2}\tag{**}
\xymatrix@C=50pt@R=30pt{
X^{\times(2n)} \ar[r]^{f^{\times(2n)}} \ar@{->>}[d]_{s_n} & X^{\times(2n)} \ar@{->>}[d]^{s_n}\\
\Alb(X) \ar[r]^{f|_{\Alb(X)}} & \Alb(X) 
}
\end{equation}
where $n\gg 1$ and $s_n:X^{\times(2n)} \to \Alb(X)$ is defined by 
$$s_n(x_1, \ldots,x_{2n})=\sum_{i=1}^n \alb_X(x_i) - \sum_{j=n+1}^{2n} \alb_X(x_j).$$
Since $s_n^*:\N^1(\Alb(X))\to \N^1(X^{\times(2n)})$ is (equivariant) injective, we conclude:
$$\iota_f=\iota_{f^{\times(2n)}}\le \iota_{f|_{\Alb(X)}}\le|\lambda|^2\le d_1(f|_{\Alb(X)})\le d_1(f^{\times(2n)})=d_1(f).$$
So the proposition is proved.
\end{proof}

\begin{theorem}\label{thm-all}
Let $f:X\to X$ be a (not necessarily separable) surjective endomorphism of a normal projective variety $X$ over an algebraically closed field $k$ of arbitrary characteristic.
Assume $d_1(f)>1$.
Then we have:

\begin{enumerate}
\item 
Set $\mathbb{K}:=\Q$ or $\R$.
Suppose $f^*D\equiv d_1(f)D$ for some $\mathbb{K}$-Cartier divisor $D$.
Then there exists another $\mathbb{K}$-Cartier divisor $D'\equiv D$ such that $f^*D'\sim_{\mathbb{K}} d_1(f)D'$.
In particular, $f^*D_1\sim_{\R} d_1(f)D_1$ for some nef $\R$-Cartier divisor $D_1$.
\item 
Suppose $f^*B\equiv qB$ for some big $\R$-Cartier divisor and $q>1$ a rational number.
Then $f$ is $q$-polarized.
\item 
Suppose $f$ is $q$-polarized. 
Then any eigenvalue of $f^*|_{H^1(X,\Q_\ell)}$ has modulus $\sqrt{q}$.
\item
Suppose $f$ is $q$-polarized (resp.\,int-amplified). Then so are the induced endomorphisms $f|_{\Alb(X)}:\Alb(X)\to \Alb(X)$ and $f|_{\mathfrak{Alb}(X)}:\mathfrak{Alb}(X)\to \mathfrak{Alb}(X)$ via the Albanese morphism and Albanese map respectively.
\end{enumerate}
\end{theorem}
\begin{proof}

We follow the notation in the proof of Proposition \ref{prop-sqrt}.

(1) Denote by $(f^{\vee})_{\mathbb{K}}:=f^{\vee}\otimes_{\Z}\mathbb{K}$ and $(X^{\vee})_{\mathbb{K}}:=X^{\vee}\otimes_{\Z}\mathbb{K}$.
Note that $d_1(f)\in \mathbb{K}$.
If $f^*D\equiv d_1(f)D$, then $f^*D-d_1(f)D=:\Delta\in (X^{\vee})_{\mathbb{K}}$.
Note that although $(X^{\vee})_{\mathbb{K}}$ is in general an infinite dimensional $\mathbb{K}$-vector space,
the Jordan canonical form of $(f^{\vee})_{\mathbb{K}}$ has all the Jordan blocks of size $\le \deg P$ (recall that $P$ is the minimal polynomial of $f^{\vee}$).
By Proposition \ref{prop-sqrt}, any diagonal entry $\lambda$ further satisfies $|\lambda| \le \sqrt{d_1(f)}<d_1(f)$.
In particular, $(f^{\vee})_{\mathbb{K}}-d_1(f)\id_{(X^{\vee})_{\mathbb{K}}}$ is invertible.
Thus, there is a $\mathbb{K}$-Cartier divisor $L\in X^{\vee}$ such that $f^*L-d_1(f)L\sim_{\mathbb{K}} \Delta$.
Then $f^*D'\sim_{\mathbb{K}} d_1(f)D'$ for $D'=D-L$.
For the second half of (1),  applying Perron-Frobenius theorem to the nef cone (cf.~\cite{Bir67}), there is a nef $\R$-Cartier divisor $D_1$ such that $f^*D_1\equiv d_1(f)D_1$; we then apply the first half of (1).

(2) Suppose $f^*B\equiv qB$ for some big $\R$-Cartier divisor and $q>1$ a rational number.
Then algebraic integer $q$ is an integer and $f^*H\equiv qH$ for some ample Cartier divisor $H$ by the norm criterion; see \cite[Propositions 1.1 and 2.9]{MZ18b}.
By (1), $f^*H'\sim_{\Q} qH'$ for some ample $\Q$-Cartier divisor $H'\equiv H'$.
Replacing $H'$ by a multiple, (2) is proved.

(3) just follows from Proposition \ref{prop-sqrt}.

(4) By the commutative diagram (\ref{diag2}) in the proof of Proposition \ref{prop-sqrt}, if $f$ is $q$-polarized (resp.\,int-amplified), then so are $f^{\times(2n)}$ and $f|_{\Alb(X)}$ (cf.~\cite[Lemma 3.10, Theorem 3.11]{MZ18b}) for polarized endomorphism and \cite[Lemma 3.4]{Men20}).
For the Albanese map $\mathfrak{alb}_X:X\dashrightarrow \mathfrak{Alb}(X)$, we may replace $X$ by the normalization of the graph of $\mathfrak{alb}_X$ and argue like above (cf.~\cite[Proof of Theorem 1.2]{CMZ20}).
So (4) is proved.
\end{proof}

\begin{remark}
In Theorem \ref{thm-all} (1), we cannot replace $d_1(f)$ by $\iota_f$ because it is possible that some eigenvalue $\lambda$ of $f^{\vee}$ is $\iota_f$.
By the intersection pairing and the projection formula, one can show in Theorem \ref{thm-all} (3) that any eigenvalue of $f^*|_{H^{2\dim(X)-1}(X,\Q_\ell)}$ has modulus $q^{\dim(X)-1/2}$.

\end{remark}

\section{Higher-dimensional Hodge Index Theorem}

In this section, we recall the higher-dimensional Hodge index theorem and give an application (cf.~Proposition \ref{prop_generalization_hodge}) which will be crucially used in Section \ref{section_proof_divisor}. 

The following lemma is known in \cite[Corollaire 3.5]{DS04} for the case of compact K\"ahler manifolds. See \cite[Lemma 4.5]{MZg20} for its full algebraic version. 
The proof below is taken from \cite[Lemma 4.5]{MZg20} with slight simplification for the convenience of the  readers.
\begin{lemma}\label{lem-hodge}
Let $X$ be a normal projective variety.
Let $L\not\equiv 0$ and $D\not\equiv 0$ be two pseudo-effective $\R$-Cartier divisors such that $L^2$ and $D^2$ are pseudo-effective (cf.~Notation \ref{not-2.1}).
Suppose $L\cdot D\equiv_w 0$.
Then $L\equiv aD$ for some $a> 0$.
\end{lemma}

\begin{proof}
We may assume $n:=\dim (X)\ge 2$.
Fix very ample Cartier divisors $A_1,\cdots,A_{n-1}$ on $X$. 
Denote by 
$$V:=\R D+\R L, \hskip 1pc \text{ and}$$
$$W:=\{x\in\textup{N}^1(X)~|~x\cdot A_1\cdots A_{n-1}=0\}$$
subspaces of $\N^1(X)$. By \cite[Lemma 3.2]{Zha16}, $L \not\equiv_w{0}$ and $D \not\equiv_w{0}$. Then $a:=L\cdot A_1\cdots A_{n-1}/D\cdot A_1\cdots A_{n-1}>0$, and $V\cap W$ is of dimension $\le 1$ and spanned by
$$\widetilde{D}:=L-aD\in V\cap W .$$ 

For each $1\le i\le n-1$, consider the following bilinear form on $\textup{N}^1(X)$:
$$q_i(x,y):=x\cdot y\cdot A_1\cdots A_{i-1}\cdot A_{i+1}\cdots A_{n-1}.$$
Choose a general normal projective surface $S$ on $X$ such that $S\equiv_w A_1\cdots A_{i-1}\cdot A_{i+1}\cdots A_{n-1}$.
Then $q_i(x,y)=x|_{S}\cdot y|_S$.
By the Hodge index theorem on $S$, $q_i$ is semi-negative on $W$.
Note that $q_i$ is semi-positive on $V$ since $L^2, D^2$ are pseudo-effective and $q_i(L, D)=0$.
Hence $q_i(\widetilde{D},\widetilde{D})=0$.
For any $w\in W$ and  $\lambda\in\mathbb{R}$, we  have $q_i(\lambda\widetilde{D}- w, \lambda\widetilde{D}-w)\le 0$. 
Then the inequality
$$q_i(w,w)-2\lambda q_i(\widetilde{D},w)\le 0$$
holds for any $\lambda\in\mathbb{R}$ and $w\in W$.
This happens only when $q_i(\widetilde{D},w)=0$ for any $w\in W$. 

Note that $W$ and $A_i$ span $\N^1(X)$ because $W$ is a hyperplane of $\N^1(X)$ and $A_i\not\in W$.
Note also that $q_i(\widetilde{D}, A_i)=0$. 
Then 
$$q_i(\widetilde{D},x)=\widetilde{D}\cdot A_1\cdots A_{i-1}\cdot x\cdot A_{i+1}\cdots A_{n-1}=0$$
for any $x\in\N^1(X)$.
This implies that $V\cap W$ is independent of the choice of each $A_i$. So 
$\widetilde{D}\cdot x_1\cdots x_{n-1}=0$
for any divisors $x_1,\cdots, x_{n-1}\in \N^1(X)$, which means $\widetilde{D}\equiv_w 0$. Note that $\widetilde{D}$ is $\mathbb{R}$-Cartier. Therefore, $\widetilde{D}\equiv 0$ by \cite[Lemma 3.2]{Zha16}.
\end{proof}

\begin{proposition}\label{prop_generalization_hodge}
Let $X$ be a normal projective variety of dimension $n$, $L$  a nef $\mathbb{R}$-Cartier divisor, and $D_1, D_2$  two non-zero effective $\mathbb{R}$-Cartier divisors on $X$. Suppose  $D_1$ and $D_2$ have no common components  and $D_1\equiv D_2$. Suppose further that $L\cdot D_1\cdot H_1\cdots H_{n-2}=0$ for some nef and big divisors $H_1,\cdots, H_{n-2}$. Then $L\equiv aD_1\equiv aD_2$ for some $a>0$ and $L^2\equiv_wD_1\cdot D_2\equiv_w 0$. In particular, $\Supp D_1$ and $\Supp D_2$ are disjoint.
\end{proposition}

\begin{proof}
For each $H_i$, we write $H_i\sim A_i+E_i$ for some ample divisor $A_i$ and effective divisor $E_i$.	 
By the assumption, the intersection of (the numerically movable) $D_1$ with any prime divisor is  pseudo-effective, hence  so is $D_1\cdot E_i$ for each $i$. Then $L\cdot D_1\cdot H_1\cdots H_{n-2}=0$ implies that 
$$L\cdot D_1\cdot A_1\cdot H_2\cdots H_{n-2}=0.$$
Repeating the same process, we have $L\cdot D_1\cdot A_1\cdots A_{n-2}=0$ and thus $L\cdot D_1\equiv_w 0$. 
By Lemma \ref{lem-hodge} and noting $D_1^2 \equiv_w  D_1 D_2$, we get $L\equiv aD_1$ ($\equiv a D_2$) with $a>0$. So we have 
$$L^2\equiv_w a^2 D_1\cdot D_2\equiv_w a L\cdot D_1\equiv_w0.$$

If $(\Supp D_1)\cap (\Supp D_2)\neq \emptyset$, then $D_1\cdot D_2$ is a non-zero effective cycle and hence we get a contradiction with $D_1\cdot D_2\not\equiv_w 0$.
\end{proof}

\section{Proofs of Theorems~\ref{main-thm-pf} and
\ref{main-thm-equiv}}\label{section_proof_divisor}

In this section, we will prove
Theorems \ref{main-thm-equiv} and  \ref{thm_divisor_case} below
which includes Theorem~\ref{main-thm-pf} as a 
simpler version and bounds the number $\sharp P_f$ in terms of the N\'eron-Severi class group rank $\rho^{\textup{ns}}(X)$, the Picard number $\rho(\widetilde{X})$ of any projective resolution $\widetilde{X}$ of $X$, and the first dynamical degree $d_1(f)$.

\begin{theorem}\label{thm_divisor_case}
Let $f:X\to X$ be a surjective endomorphism of a normal projective variety.
Let $\widetilde{X} \to X$ be any projective resolution.
Then
$$\sharp P_f\le \textup{max}\{(1+\sqrt{d_1(f)})^2+a, (\rho^{\textup{ns}}(X)+1)\cdot(\lfloor\frac{\rho(\widetilde{X})}{2}\rfloor+1)\},$$ 
where $a$ is a positive constant depending only on $X$. 
In particular, 
for $n\gg 1$, 
$$\sharp P_{f^n}\le(d_1(f))^n\cdot (1+o(1)) .$$
\end{theorem}

We now prepare to prove Theorem~\ref{thm_divisor_case}.

\begin{notation}\label{notation_proof}
Let $f:X\to X$ be a surjective endomorphism on a normal projective variety $X$ of dimension $n$.
We fix a projective resolution $$\sigma:\widetilde{X}\to X$$
with $E_1,\cdots, E_k$ its exceptional prime divisors.  
For $D \in P_f$, denote by
$\widetilde{D}:= \sigma_*^{-1}(D)$, the 
proper transform of $D$.
Denote by $$\widetilde{P_f}:=\{\widetilde{D} = \sigma_*^{-1}(D)\,|\,D\in P_f\}.$$ 
By Perron-Frobenius theorem (cf.~\cite{Bir67}), fix a nef $\mathbb{R}$-Cartier divisor $L_f\not\equiv 0$ on $X$ with $$f^*L_f\equiv d_1(f) L_f.$$
Denote by $$\widetilde{L_f}:=\sigma^*L_f.$$ 
\end{notation}

\begin{lemma}\label{lem_weak_sdd}
$L_f\cdot D\equiv_w 0$ for any $D\in P_f$.
\end{lemma}
\begin{proof}
Note that $(f|_D)^*(L_f|_D)\equiv d_1(f)(L_f|_D)$ and $d_1(f|_D)<d_1(f)$.
Then $L_f|_D\equiv 0$ and hence $L_f\cdot D\equiv_w 0$.
\end{proof}

\begin{lemma}\label{lem_N_disjoint}
There exist at least $2\lfloor\sharp P_f/(\rho^{\textup{ns}}(X)+1)\rfloor$ disjoint prime divisors $\widetilde{D}_u$ in $\widetilde{P_f}$ with $\widetilde{L_f}\cdot\widetilde{D}_u\equiv_w 0$.
\end{lemma}

\begin{proof}
Let $m:=\lfloor\sharp P_f/(\rho^{\textup{ns}}(X)+1)\rfloor$.
Then we have $m$ pairwise disjoint subsets $S_1,\cdots, S_m$ of $\widetilde{P_f}$ with $\sharp S_t=\rho^{\textup{ns}}(X)+1$ for each $1\le t\le m$.

Note that $\rho(\widetilde{X})=\rho^\textup{ns}(X)+k$ where $k$ is the number of $\sigma$-exceptional prime divisors (cf.~\cite[Lemma 18]{Kol18}). 
By the assumption, we have the equality
$$\sharp (S_t\cup \{E_1,\cdots, E_k\}) =  \rho(\widetilde{X})+1.$$ 
For each $t$, we then have a numerical equivalence of two non-zero effective divisors which are combinations of prime divisors in $S_t$ and $\sigma$-exceptional prime divisors such that they have no common component. We write them in the following form
$$\sum_{i\in I_t} a_i \widetilde{D}_i+ \widetilde{E}_t\equiv \sum_{j\in J_t} b_j \widetilde{D}_j+\widetilde{F}_t$$
where $I_t$ and $J_t$ are two (possibly empty) disjoint subsets of $S_t$, $a_i>0$, $b_j>0$, and $\widetilde{E}_t, \widetilde{F}_t$ are effective $\sigma$-exceptional divisors. Here without ambiguity, we identify  $\widetilde{D}_i$ with the index $i$ in $I_t$. 

Let $H$ be an ample divisor on $X$.
By Lemma \ref{lem_weak_sdd} and the projection formula,
$$\widetilde{L_f}\cdot (\sum_{i\in I_t} a_i \widetilde{D}_i+\widetilde{E}_t)\cdot (\sigma^*H)^{n-2}=L_f\cdot (\sum_{i\in I_t} a_i D_i)\cdot H^{n-2}=0$$
where $D_i=\sigma_*(\widetilde{D_i})\in P_f$.
Applying Proposition \ref{prop_generalization_hodge}, for each $t$, 
we may assume (after some scalar multiplication)
$$\widetilde{L_f}\equiv \sum_{i\in I_t} a_i \widetilde{D}_i+ \widetilde{E}_t\equiv \sum_{j\in J_t} b_j \widetilde{D}_j+\widetilde{F}_t$$
and we also have
$$\widetilde{L_f}^2\equiv_w \widetilde{L_f}\cdot (\sum_{i\in I_t} a_i \widetilde{D}_i+ \widetilde{E}_t)\equiv_w \widetilde{L_f}\cdot (\sum_{j\in J_t} b_j \widetilde{D}_j+\widetilde{F}_t)\equiv_w 0.$$
This further implies that 

\begin{equation}\label{eq1}
	\widetilde{L_f}\cdot \widetilde{D}_i\equiv_w \widetilde{L_f}\cdot \widetilde{D}_j\equiv_w \widetilde{L_f}\cdot \widetilde{E} \equiv_w 0 
\end{equation}
for any $i\in I_t$, $j\in J_t$ and any prime divisor $\widetilde{E}$ in $\widetilde{E}_t$ and $\widetilde{F}_t$.

We claim that $I_t$ and $J_t$ are non-empty.
Suppose the contrary, for example, that $I_t$ is empty.
Then $\widetilde{L_f}\equiv \widetilde{E}_t\neq 0$.
However, this contradicts with the negativity lemma (cf.~\cite[Lemma 3.39]{KM98}).

Consider the pair $(I_t, J_{t'})$ for any $t, t'$.
Note that
\begin{equation}\label{eq2}
\sum_{i\in I_t} a_i \widetilde{D}_i+ \widetilde{E}_t\equiv \sum_{j\in J_{t'}} b_j \widetilde{D}_j+\widetilde{F}_{t'} \, (\equiv \widetilde{L_f})
\end{equation}

We claim that $\widetilde{E}_t$ and $\widetilde{F}_{t'}$ have no common components.
Readjust it a bit in the following way, where $\Delta$ is a $\sigma$-exceptional effective $\mathbb{R}$-divisor, 
$$\sum_{i\in I_t} a_i \widetilde{D}_i+ (\widetilde{E}_t-\Delta)\equiv \sum_{j\in J_{t'}} b_j \widetilde{D}_j+(\widetilde{F}_{t'}-\Delta),$$
so that $\widetilde{E}_t-\Delta$ and $\widetilde{F}_{t'}-\Delta$ are effective but with no common components.
Note that $$\widetilde{L_f}\cdot (\sum_{i\in I_t} a_i \widetilde{D}_i+ (\widetilde{E}_t-\Delta))\equiv_w 0.$$
By Proposition \ref{prop_generalization_hodge}, 
for some $e>0$,
we have $$e\widetilde{L_f}\equiv \sum_{i\in I_t} a_i \widetilde{D}_i+ (\widetilde{E}_t-\Delta) .$$
Then $(1-e)\widetilde{L_f}\equiv \Delta$,
and hence $e\le 1$ since both $\widetilde{L_f}$ and $\Delta$ are pseudo-effective and $X$ is projective.
By the negativity lemma again (cf.~\cite[Lemma 3.39]{KM98}), $e=1$ and $\Delta \equiv 0$, so $\Delta = 0$ and the claim is proved.
By Proposition \ref{prop_generalization_hodge} and the display (\ref{eq2}) above and since 
$\widetilde{L_f} \cdot \widetilde{L_f} \equiv_w 0$,
we have $\widetilde{D}_i\cap \widetilde{D}_j=\emptyset$
for any $i\in I_t$ and $j\in J_{t'}$.

The same argument works for the pairs $(I_t, I_{t'})$ and $(J_t, J_{t'})$ with $t\neq t'$.
In particular, we may pick one prime divisor in $I_t$ and one prime divisor in $J_t$ for each $t$, and all these $2m$ prime divisors are pairwise disjoint.
This completes the proof.
\end{proof}

With $\sharp P_f$ large enough, we can construct an $f$-equivariant fibration over a curve.

\begin{proposition}\label{lem_reduction_nef}
Suppose $\sharp  P_f\ge (\lfloor\frac{\rho(\widetilde{X})}{2}\rfloor+1)\cdot (\rho^{\textup{ns}}(X)+1)$. 
Then there exists a fibration $\widetilde{\pi}:\widetilde{X}\to C$ onto a smooth curve $C$ such that the following hold.
\begin{enumerate}
\item  $\widetilde{L_f}\equiv \widetilde{\pi}^*H_C$ for some ample $\R$-Cartier divisor $H_C$ on $C$.
\item  The fibration $\widetilde{\pi}$ factors through a fibration $\pi:X\to C$ such that $L_f\equiv \pi^*H_C$.
\item There exists a $d_1(f)$-polarized endomorphism $g:C\to C$ such that $\pi\circ f=g\circ \pi$.
\end{enumerate}
\end{proposition}

\begin{proof}
(1) By Lemma \ref{lem_N_disjoint}, there are $\rho(\widetilde{X})+1$ pairwise disjoint prime divisors $\widetilde{D}_u$ in $\widetilde{P_f}$ with $\widetilde{L_f}\cdot \widetilde{D}_u\equiv_w 0$. 
By \cite[Theorem 1.1]{BPS16}, there is a fibration $\widetilde{\pi}:\widetilde{X}\to C$ onto a smooth projective curve such that $\widetilde{\pi}(\widetilde{D}_u)$ is a point for each $u$ and $\widetilde{D}_u=\widetilde{\pi}^{-1}(\widetilde{\pi}(\widetilde{D}_u))$ for $1\le u\le 3$. 
Now $\widetilde{D}_1\equiv a\widetilde{D}_2$ for some $a>0$.
By Proposition \ref{prop_generalization_hodge}, $\widetilde{L_f}\equiv t \widetilde{D}_1$ for some $t>0$.
Note that $\widetilde{\pi}^{*}(\widetilde{\pi}(\widetilde{D}_1))=b\widetilde{D}_1$ for some $b>0$.
Letting $H_C:=t\widetilde{\pi}(\widetilde{D}_1)/b$,
(1) follows.

(2) follows from (1) and Lemma \ref{lem_nef_descend}. 

(3) For every curve $B$ lying in a fibre of $\pi$, by (2) and the projection formula, we have
$$\pi_*(f_*B)\cdot H_C=f_*B\cdot L_f= B \cdot f^*L_f = d_1(f) B\cdot L_f=d_1(f) \pi_*B\cdot H_C=0$$
So every fibre of $\pi$ will also be contracted by $\pi\circ f$. 
By the rigidity lemma (cf.~\cite[Lemma 1.15]{Deb01}), $\pi\circ f$ factors through $\pi$ and thus $f$ descends to a surjective endomorphism $g=f|_C$. Moreover, if we assume $g^*H_C\equiv qH_C$, then
$$f^*L_f\equiv f^*\pi^*H_C=\pi^*g^*H_C\equiv q\pi^*H_C\equiv qL_f,$$
which implies $q=d_1(f)$. Hence $g$ is $d_1(f)$-polarized. This proves (3).
\end{proof}

With all the preparations done, 
we are ready to prove  Theorems \ref{thm_divisor_case}, \ref{main-thm-pf} and \ref{main-thm-equiv}.

\begin{proof}[Proof of ~Theorem~\ref{thm_divisor_case}]
If $d_1(f)=1$, then $\sharp P_f=0$ and the theorem is trivial.
From now on we assume that $d_1(f)>1$.

Suppose $\sharp P_f\ge (\rho^{\textup{ns}}(X)+1)\cdot(\lfloor\frac{\rho(\widetilde{X})}{2}\rfloor+1)$. 
By Proposition \ref{lem_reduction_nef}, there is an $f$-equivariant fibration $\pi:X\to C$ onto a smooth curve $C$ such that $g:=f|_C$ is $d_1(f)$-polarized.
Let $\mathcal{B}$ be the finite set of prime divisors in the reducible fibres of $\pi$.
By the same argument as in the proof of Proposition \ref{prop-pf-bound}, we have an injection $\varphi:P_{f}\backslash \mathcal{B}\to \Fix(g)$, sending $D$ to $\pi(D)$. 
Let $a:=\sharp \mathcal{B}$.
By Lemma \ref{lem-lefschetz}, we can conclude the theorem:
$$\sharp P_f\le \sharp \Fix(g)+a\le (1+\sqrt{d_1(f)})^2+a .$$
\end{proof}

\begin{proof}[Proof of Theorem \ref{main-thm-pf}]
Apply Theorem \ref{thm_divisor_case} to the 
normalization of $X$ (cf.~Lemma \ref{bir_invariant_case}).
\end{proof}

\begin{proof}[Proof of Theorem \ref{main-thm-equiv}]
By taking normalization of $X$, we may assume that  $X$ is normal (cf.~Lemmas \ref{bir_invariant_case} and \ref{bir_totally_periodic_case}).

``$(2)\Rightarrow (1).$''
By \cite[Theorem 5.1]{Fak03}, $\Per(g)$ is Zariski dense in $C$.
Since $\pi$ is a fibration, it has irreducible and reduced general fibres.
Therefore $X_c:=\pi^{-1}(c)=\pi^*c$ is an $f$-periodic prime divisor for general $c\in \Per(g)$.
Suppose $f^s(X_c)=X_c$.
Then $d_1(f^s|_{X_c})=d_1(f^s|_\pi)<d_1(f^s)$ by the assumption.
So $X_c\in P_{f^\infty}$ and $\sharp P_{f^\infty}$ is infinite.

``$(1)\Rightarrow (2).$''
After a sufficient iteration of $f$, we may assume that $\sharp P_{f}\gg 1$. 
Applying Proposition \ref{lem_reduction_nef}, there is an $f$-equivariant fibration $\pi:X\to C$ onto a smooth curve $C$ such that $g:=f|_C$ is $d_1(f)$-polarized.
We may also assume that $D=\pi^{-1}(\pi(D))$ for some $D\in P_f$.
Then $d_1(f|_\pi)=d_1(f|_D)<d_1(f)$, which completes the first part of our theorem.

Now we prove the finiteness of $P_{f^{\infty}}^{-1}$. Suppose the contrary that $P^{-1}_{f^{\infty}}$ is infinite.
By the above equivalent condition, there is an $f$-equivariant (after iteration) fibration $\pi:X\to C$ onto a smooth curve $C$ such that $g:=f|_C$ is $d_1(f)$-polarized. 
Let $\mathcal{B}$ be the finite set of prime divisors in the reducible fibres of $\pi$.
By the same argument as in the proof of Proposition \ref{prop-pf-bound} and \cite[Lemma 7.5]{CMZ20}, we have an injection $\varphi:P^{-1}_{f^{\infty}}\backslash \mathcal{B}\to P^{-1}_{g^{\infty}}$, sending $D$ to $\pi(D)$. 
Note that there are at most finitely many $g^{-1}$-periodic points (cf.~\cite[Corollary 3.8]{MZ20}).
So the theorem is proved.
\end{proof}

\begin{corollary}
Let $f: X \to X$ be a surjective endomorphism of a normal projective variety $X$  with $\sharp P_{f^{\infty}}=\infty$.
Then $D$ is either $\Q$-Cartier or $f^{-1}$-periodic for any $D\in P_{f^{\infty}}$.
\end{corollary}

\begin{proof}
We apply Theorem \ref{main-thm-equiv} and use the notation there.
If $D=\pi^{-1}(\pi(D))$, then $D$ is $\Q$-Cartier.
Otherwise, $\pi^{-1}(\pi(D))$ is reducible.
Denote by $\Sigma$ the finite subset of $C$ over which $\pi$ has reducible fibres.
Then $g^{-1}(\Sigma)=\Sigma$ by \cite[Lemma 7.4]{CMZ20}.
Therefore $D$ is an irreducible component of $\pi^{-1}(\Sigma)$ and hence $f^{-1}$-periodic.
\end{proof}

\end{document}